\documentclass[11pt]{article}
\usepackage{amsmath}
\usepackage{amsfonts} \makeatother
\usepackage{amsfonts,amsmath,amssymb,mathrsfs}
\usepackage{cite}
\usepackage{enumitem}
\usepackage{amsthm}
\usepackage[colorlinks=true,urlcolor=blue,
citecolor=red,linkcolor=blue,linktocpage,pdfpagelabels,
bookmarksnumbered,bookmarksopen]{hyperref}

\newcommand{\cL}{{\mathcal L}}

\numberwithin{equation}{section}
\newtheorem{theorem}{Theorem}[section]
\newtheorem{lemma}[theorem]{Lemma}
\newtheorem{proposition}[theorem]{Proposition}
\newtheorem{remark}{Remark}[section]

\newtheorem{corollary}[theorem]{Corollary}

\newtheorem{definition}[theorem]{Definition}
\newcommand{\be}{\begin{equation}}
\newcommand{\ee}{\end{equation}}
\newcommand\bes{\begin{eqnarray}}
\newcommand\ees{\end{eqnarray}}
\newcommand{\bess}{\begin{eqnarray*}}
\newcommand{\eess}{\end{eqnarray*}}

\newcommand\inv{\frak i}

\DeclareMathOperator*{\einf}{ess\,inf}
\DeclareMathOperator{\dist}{dist}

\newcommand{\id}{\textnormal{id}}

\newcommand\eps{\varepsilon}

\def\R{\mathbb{R}}
\def\Z{\mathbb{Z}}
\def\N{\mathbb{N}}
\def\cN{\mathcal{N}}

\usepackage{authblk}
\usepackage{orcidlink}
\author[1]{O. Cabrera}
\author[2]{S. Cingolani~\orcidlink{0000-0002-3680-9106}}
\author[3]{T. Weth~\orcidlink{0000-0001-5347-8057}}
\affil[1,3]{Institut für Mathematik, Goethe-Universität Frankfurt am Main, Robert-Mayer-Str. 10, D-60629 Frankfurt am Main, Germany.}
\affil[2]{Dipartimento di Matematica, Universit\`{a} degli Studi di Bari Aldo Moro, Via E. Orabona 4, 70125, Bari, Italy}

\title{A new framework for Ljusternik-Schnirelmann theory and its application to planar Choquard equations}

\begin{document}

	\date{}
	
	\maketitle

\begin{abstract}
  {We consider the planar logarithmic Choquard equation
    $$
    - \Delta u + a(x)u + (\log|\cdot| \ast u^2)u = 0,\qquad \text{in } \R^2
    $$
    in the strongly indefinite and possibly degenerate setting where no sign condition is imposed on the linear potential $a \in L^\infty(\R^2)$. In particular, we shall prove the existence of a sequence of high energy solutions to this problem in the case where $a$ is invariant under $\Z^2$-translations.
    The result extends to a more general $G$-equivariant setting, for which we develop a new variational approach which allows us to find critical points of Ljusternik-Schnirelmann type. In particular, our method resolves the problem that the energy functional $\Phi$ associated with the logarithmic Choquard equation is only defined on a subspace $X \subset H^1(\R^2)$ with the property that $\|\cdot\|_X$ is not translation invariant. The new approach is based on a new $G$-equivariant version of the Cerami condition and on deformation arguments adapted to a family of suitably constructed scalar products $\langle \cdot, \cdot \rangle_u$, $u \in X$ with the $G$-equivariance property $\langle g \ast v , g \ast w \rangle_{g \ast u} = \langle v , w \rangle_u.$
}
	\end{abstract}		

	\medskip



\section{Introduction}

In this paper we are concerned with the logarithmic Choquard equation
\begin{equation}\label{eq:choquard_equation}
	- \Delta u + a(x) u + (\log|\cdot| \ast u^2)u = 0, \qquad \text{in } \R^2,
\end{equation}
with potential $a \in L^{\infty}(\R^2)$, which arises as a reduced equation for the Schrödinger-Poisson system\footnote{In the literature, the system is also called Schrödinger-Newton system or Schrödinger-Maxwell system.}
\begin{equation}\label{eq:schroedinger-newton-system}
\left\{
\begin{aligned}
- \Delta u &+ a(x)u +  \gamma \,w u = 0,\\
&\Delta w = u^2
\end{aligned}
 \right.
\qquad \text{in $\R^2$}
\end{equation}
after inverting the second equation by means of convolution with the Newtonian potential $x \mapsto \Phi_2(x)=\frac{1}{2\pi} \log |x|$. Here, $\gamma>0$ is a parameter which can be adjusted after rescaling. The Schr\"odinger-Newton
system has been derived in $\mathbb{R}^3$ by R. Penrose in \cite{Pe1, Pe2} in his description of the self-gravitational collapse of a quantum mechanical
system (see also \cite{lieb}). In this case, inversion of the second equation via convolution with the Coulomb potential $\Phi_3(x)= \frac{1}{4 \pi |x|}$
 leads to the classical Choquard equation, which, in the case $a \equiv \lambda>0$, was introduced by Pekar \cite{pekar} in 1954 in a quantum mechanical model of a polaron at rest and then later by Choquard in 1976 in a context of self-trapping of an electron. For more information and a comprehensive list of references, we refer the reader to the survey \cite{mvs}.	

We stress that, differently from the Coulomb potential, the Newton potential in $\R^2$ is sign-changing and singular both at at zero and at infinity. 
As a consequence, the planar nonlocal equation $(\ref{eq:choquard_equation})$ presents special challenges and requires a tailor-made functional-analytic setting which takes both the translation invariance of the convolution term and the growth of the logaritmic convolution potential at infinity into account. In particular, the applicability of variational tools is not straightforward, because the usual Sobolev space $H^1(\R^2)$ is no suitable environment for a direct variational approach to (\ref{eq:choquard_equation}), while the logarithmic convolution term also leads to nonstandard variational geometries.

Due to these difficulties, a rigorous study of (\ref{eq:schroedinger-newton-system}) and (\ref{eq:choquard_equation}) remained elusive for some time.
In the literature, apart from some numerical results in \cite{HMT},
first existence and uniqueness results of spherically symmetric solutions of $(\ref{eq:choquard_equation})$
were proved by Stubbe and Vuffray \cite{CSV}, for $a \equiv 1$, using
shooting methods for the associated ODE system.
 In 
\cite{stubbe}, Stubbe studied (\ref{eq:choquard_equation}) with a positive constant potential $a$ by introducing a 
smaller Hilbert space $X \subset H^1(\R^2)$
by setting 
\begin{gather*}
	X= \left\{u \in H^1(\R^2)\mid \int_{\R^2}\log(1+|x|) |u(x)|^2 \,dx <
	\infty \right\},
\end{gather*}
endowed with the norm
\begin{equation}\label{eq:X_norm}
	\| u \|_X^2 := \| u \|_{H^1(\R^2)}^2 + |u|_*^2 \qquad |u|_*^2 := \int_{\R^2} \log(1+|x|)u^2(x) dx.
\end{equation}
Using rearrangement inequalities, Stubbe proved that there is a unique positive and radially symmetric ground state to 
 $(\ref{eq:choquard_equation})$. Moreover, in \cite{Masaki2} Masaki proved that the Cauchy problem associated to the time-dependent Schrödinger-Poisson system is globally well-posed in the space $X$, see also \cite{Masaki}.

Even if the space $X$ provides a reasonable variational framework, its
norm does not detect the translation invariance of the convolution term in (\ref{eq:choquard_equation}). Moreover, the quadratic part of the energy functional associated to 
 $(\ref{eq:choquard_equation})$
is not coercive on $X$.  These difficulties
enforced the implementation of new variational ideas and estimates to
treat the nonlinear Schr\"odinger equation with nonlocal nonlinearities
involving logarithmic type convolution potential.
Inspired by the variational framework in \cite{stubbe}, the second and third authors  showed in \cite{CW} that, when $a$ is a positive $\mathbb{Z}^2$-periodic potential, the equation \eqref{eq:choquard_equation} admits a positive ground state, which is obtained by minimizing the corresponding energy functional over the Nehari set. In fact, this existence result is proved in \cite{CW} for a more general version of (\ref{eq:choquard_equation}) with a nonlinear local term, i.e., 
for
\begin{equation}\label{eq:choquard_equation2}
- \Delta u +  a(x)u + \gamma (\log|\cdot| \ast u^2)u = b |u|^{p-2}u \qquad \text{in $\:\R^2$},
\end{equation}
where  $p \geq 4$,  $b \geq 0$. In particular, \cite{CW} provides a variational characterization of least energy solutions of \eqref{eq:choquard_equation2}, and it is shown that these solutions do not change sign and are unique up to translations if $a$ is a positive constant. Successively, the sharp asymptotics and nondegeneracy of such ground state energy solutions has been proved by Bonheure, Van Schaftingen and the second author \cite{BoCiVa}.

Since then, the planar logarithmic Choquard equation (\ref{eq:choquard_equation}) and its generalization~(\ref{eq:choquard_equation2}) have received increasing attention and have been studied, in particular, in the more recent papers \cite{dw,c,BVS, CiJe,PWZ,liu-radulescu-tang-zhang} under various different assumptions on $a$ and $p$.

However, up to now no attempt has been made to analyze (\ref{eq:choquard_equation}) in the {\em strongly indefinite case} where no sign condition on the potential $a$ is imposed. The main purpose of the present paper is to develop a new abstract variational approach (\ref{eq:choquard_equation}) which leads to existence and multiplicity results in the strongly indefinite case, while its unified nature also allows to directly recover existing results for the case $a>0$.

Let us briefly mention the main conceptual difficulty in the case where no sign condition is imposed on $a$. In this case $0$ may be contained in the essential spectrum of the Schrödinger operator $S:=-\Delta + a$, so $S$ may not even be a Fredholm operator between classical Sobolev spaces and standard nonlinear analytic methods fail. In this paper, we shall see that an effect of {\em nonlinear compactness} (up to translations) outweighs the issues related to the linear operator $S$.

In order to explain our variational approach, let us introduce, for $a \in L^\infty(\R^2)$, the quadratic form
$$
u \mapsto q_a(u)= \int_{\R^2} |\nabla u|^2\,dx + \int_{\R^2}a(x)u^2\,dx 
$$
and the $4$-homogeneous functional
$$
u \mapsto V_0(u)= \int_{\R^2} \int_{\R^2}
\log|x-y|  u^2(x)u^2(y)\,dx dy
$$
which is easily seen to be well-defined on the space $X$ defined above. 
Then weak solutions of $(\ref{eq:choquard_equation})$ can be found as critical points of the energy functional
$$
\Phi: X \to \R, \qquad \Phi(u) = \frac{1}{2}q_a(u) + \frac{1}{4} V_0(u),
$$
and they are all contained in the associated Nehari manifold\footnote{In general, the set $\cN$ is not a manifold, but we still use this name as it is common in the literature.} 
$$
\cN:=  \{u \in X \setminus \{0\} \::\: \Phi'(u)u = 0\} = \{u \in X \setminus \{0\} \::\: q_a(u) = -V_0(u)\}.
$$
In the general case where $a$ may be negative or sign-changing, both $q_a$ and $V_0$ are neither bounded from above nor from below, and therefore the set $\cN$ may have a complicated shape. In fact, it is useful to split into the subsets
$$
\cN_\pm := \{u \in \cN\::\:  \pm V_0(u) > 0 \} \quad \text{and}\quad \cN_0:= \{u \in \cN\::\: V_0(u) = 0\}.
$$
We shall see that $\cN_\pm$ are $C^2$-submanifolds of codimension 1 of $X$, and we will set up variational principles to find critical points of $\Phi$ in
$\cN_-$, leaving the study of $\cN_+$ for future work. The main difficulties we face when investigating $\cN_\pm$ is the fact that neither one of these submanifolds is complete in general, and $\Phi$ is not coercive on $\cN_\pm$. We note that this stands in striking constrast to common applications of the Nehari manifold method. We also note that in general there may exist functions $u \in X \setminus \{0\}$ with $q_a(u)=V_0(u)=0$, so that the full ray $\{tu \::\: t>0\}$ is contained in $\cN_0$ in this case.   

A further difficulty, already present in the case of positive potential functions $a \in L^\infty(\R^2)$, is the lack of compactness caused by the translation invariance of the convolution functional $V_0$. To deal with this problem, we follow a similar strategy as in \cite{CW} and assume that $a$ is invariant under a subgroup $G$ of euclidean motions in $\R^2$. The group $G$ then also acts in a canonical way on $X$ by defining $g* u = u \circ g^{-1}$ for $u \in X$, $g \in G$, and the functional $\Phi$ remains invariant under this action. In particular, in the special case where $G= \Z^2$, we are dealing with a translation-invariant problem with lack of compactness.

In this context, a further problems arises. First, while $\Phi$ remains $G$-invariant, the norm of $X$ is not $G$-invariant. In particular, every critical point $u \in X$ of $\Phi$ is contained in an orbit $\{g\ast u\::\: g \in G\}$ of critical points which is unbounded in the norm of $X$. As a consequence, in addition to lack of compactness, we also have a lack of boundedness of Palais-Smale or Cerami sequences of $\Phi$ in $X$.

Here we recall that a sequence $(u_n)_n$ in $X$ is called a Cerami sequence for $\Phi$ at the level $c \in \R$ 
if $\Phi(u_n) \to c$ and $\|\Phi(u_n)'\|_{X^*} (1+ \|u_n\|_X) \to$ as $n \to \infty$, where $X^*$ denotes the topological dual of $X$.

In the present paper, in order to prove existence and multiplicity results in the strongly indefinite setting, we implement some new advances in $G$-equivariant critical point theory based on a compactness condition relative to the $G$-action for a suitably adapted $G$-invariant notion of Cerami sequences. 
More precisely, we develop an abstract variational method for $G$-invariant $C^{1,1}$-functionals defined on submanifolds of a generic Hilbert space $X$ which do not satisfy a suitable compactness condition with respect to $\|\cdot\|_X$, but are better adapted to a family of equivalent, but not uniformly equivalent norms $\|\cdot\|_u$, $u \in X \setminus \{0\}$.   
The typical situation occurs when the energy functional is invariant under the orthogonal action of a noncompact topological group on $X$, but the norm $\|\cdot\|_X$  on $X$ is not $G$-invariant. In this case the family of norms $\|\cdot\|_u$ should be chosen such that the following $G$-invariance property holds:
\begin{equation}
  \label{eq:G-invariant-norm-intro}
\|g\ast v\|_{g\ast u} = \|v\|_u \qquad \text{for all $u \in X \setminus \{0\}$, $v \in X$.}
\end{equation}
In the application to the logarithic Choquard equation, the construction of this $G$-invariant family of norms uses the notion of a generalized barycenter $\beta(u) \in \R^N$ of a function $u \in X$. Such a generalized barycenter has been introduced independently in \cite{bw,cp}, and it allows us to define 
$$
\|v\|_u := \|v(\cdot + \beta(u))\|_X \qquad \text{for $u, v \in X$.}
$$
The equivariance of the generalized barycenter with respect to euclidean motions then gives rise to the abstract property~(\ref{eq:G-invariant-norm-intro}). 

At this stage, we wish to point out the papers \cite{BVS,liu-radulescu-tang-zhang} which contain different strategies to recover a translation-invariant setting in the space $H^1(\R^N)$ by modifying the convolution term in (\ref{eq:choquard_equation}) and then applying an approximation procedure. In contrast, we directly work with the original equation (\ref{eq:choquard_equation}) in the space $X$ and therefore avoid the analysis of approximations. This is particularly important in the context of multiplicity results since in general it is difficult to distinguish limits of solutions to approximated problems. 

Let us now state or main existence and multiplicity results on the indefinite Choquard nonlinear equation. We start with a result where $G= \Z^2$.

\begin{theorem}\label{thm:multiplicity_translation}
	Let $a \in L^\infty(\R^2)$ be $\Z^2$-invariant. Then the equation \eqref{eq:choquard} has a sequence $(\pm u_n)_n$ of nontrivial solution pairs with $\Phi(\pm u_n) \to \infty$ as $n \to \infty$. Moreover, if $\einf_{\R^2} a > 0 $, then 
	$$ \Phi(\pm u_1) = \inf\{ \Phi(u) : u \in X \setminus \{0\}, \, \Phi'(u) = 0\}. $$ 
\end{theorem}

Our next result is devoted to the radial case with $G = O(2)$. 

\begin{theorem}\label{thm_multiplicity_radial}
	Let $a \in L^{\infty}(\R^2)$ be radially symmetric. Then the Choquard equation
	\begin{equation*}
		-\Delta u + a(x)u + (\log|\cdot| \ast u^2)u = 0 \quad x \in \R^2,
	\end{equation*}
	admits a sequence of infinitely many radially symmetric solution pairs $(\pm u_n)_n$ in $X$ with \mbox{$\Phi(\pm u_n) \to \infty$}. Moreover, if $\einf_{\R^2} a > 0 $, then 
	$$ \Phi(\pm u_1) = \inf\{ \Phi(u) : u \in X \setminus \{0\}, \, u \text{ is radially symmetric and } \Phi'(u) = 0\}. $$ 
\end{theorem}

Furthermore we can produce infinitely many sign-changing nonradial solutions to the equation 
$(\ref{eq:choquard_equation})$. For this we fix $m \in \mathbb{N}$, we let $g \in O(2)$ be a rotation of angle $\pi/m$, and we consider the cyclic group $G := \{ g^{j}\::\: j=0,2m-1\}$ which acts on $X$ as follows
\begin{equation}\label{eq:sign_changing_symmetry}
	g^j \ast u(x) = (-1)^j u(g^{-j} x) \quad \text{a.e. in } \R^2.
\end{equation}
We then have the following:

\begin{theorem}\label{thm_multiplicity_nonradial}
	Let $a \in L^{\infty}(\R^2)$ be $G$-invariant, which holds, in particular, if $a$ is radially symmetric. Then the Choquard equation
	\begin{equation*}
		-\Delta u + a(x)u + (\log|\cdot| \ast u^2)u = 0 \quad x \in \R^2,
	\end{equation*}
	admits a sequence of infinitely many solution pairs $(\pm u_n)_n$  in $X$ with $\Phi(\pm u_n) \to \infty$ which obey \eqref{eq:sign_changing_symmetry}. In particular, these solutions are sign-changing and nonradial. Moreover, if $\einf_{\R^2} a > 0 $, then 
	$$ \Phi(\pm u_1) = \inf\{ \Phi(u) : u \in X \setminus \{0\}, \, u \text{ is } G\text{-invariant, and } \Phi'(u) = 0\}. $$
\end{theorem}

Finally, we let $G$ be generated by the reflection $h$ at the $x_1$-axis and translations by $ e_1 \Z := \{ (z,0) \in \R^2 : z \in\ Z\}$. Moreover, we let $G_0:= \{\id, h\}$ be the normal subgroup generated by $h$. Then a $G_0$-invariant function $u \in X$ satisfies
\begin{equation}\label{eq:c3symmetries}
u(x_1,-x_2) = u(x_1,x_2) \qquad \text{for all $x \in \R^2$.}
\end{equation}
\begin{theorem}\label{thm:multiplicity_reflection}
  Let $a \in L^\infty(\R^2)$ be $G$-invariant, i.e. we have 
  $$
  a(x_1,-x_2) = a(x_1,x_2) \quad \text{und} \quad a(x_1+k,x_2) = a(x_1,x_2)\qquad \text{for all $x \in \R^2$, $k \in \Z$.}
  $$
  Then the Choquard equation
	\begin{equation*}
		-\Delta u + a(x)u + (\log|\cdot| \ast u^2)u = 0 \quad x \in \R^2,
	\end{equation*}
        admits a sequence of infinitely many solution pairs $(\pm u_n)_n$  in $X$ with $\Phi(\pm u_n) \to \infty$, which obey \eqref{eq:c3symmetries}. Moreover, if $\einf_{\R^2} a > 0 $, then 
        $$ \Phi(\pm u_1) = \inf\{ \Phi(u) : u \in X \setminus \{0\}, \, u \text{ is } G_0\text{-invariant, and } \Phi'(u) = 0\}. $$
\end{theorem}

This paper is organized as follows. In Section 
\ref{sec:LS theory}
we introduce in detail the abstract theory and state our main abstract result
Theorem \ref{c-k-to-infty}.

In Section \ref{subsec:cerami-case}, we prove  Theorem \ref{c-k-to-infty}. We also discuss the variant case dealing with functions which are $G_0$-invariant for a normal subgroup $G_0 \subset G$ and state the analogue version of Theorem \ref{c-k-to-infty} in this new setting (see Theorem \ref{c-k-to-infty-variant}).

Then the Section \ref{sec:choquard} is dedicated to the indefinite nonlinear Choquard equation.  We put into place the pieces necessary to apply Theorems \ref{c-k-to-infty} and \ref{c-k-to-infty-variant}. In particular we construct the generalized barycenter map,  introduce the symmetries we will work with and construct a $G$-invariant metric.

In Section \ref{sec:choquard} we applied the abstract theory developed in Section \ref{sec:LS theory} to produce infinitely many solutions for the indefinite Choquard equation under different symmetries.


\section{A new version of G-equivariant LS theory}\label{sec:LS theory}
The purpose of this section is to establish an abstract multiplicity result in a new setting for $G$-equivariant critical point theory.

\subsection{The main result in an abstract setting}
\label{sec:abstr-sett-main}
We consider an Hilbert space $X$ with scalar product $\langle \cdot, \cdot \rangle_X$,  induced by the norm $\|\cdot\|_X$. 
We make the following general assumptions for a given family of scalar products $\langle \cdot,\cdot \rangle_u$, $u \in X \setminus \{0\}$ and the induced norms $\|\cdot\|_u$.
\begin{enumerate}[label=(M\arabic*),ref=(M\arabic*)]
	\item \label{M1} $\| \cdot \|_u$ is equivalent to $\| \cdot \|_X$ for all $u \in X \setminus \{0\}$, and this equivalence is locally uniform. More precisely, for all $u \in X \setminus \{0\}$, there is a neighborhood $U \subset X \setminus \{0\}$ of $u$ and constant $C=C(u)$ such that  
	\begin{equation}
		\label{local-uniform-equivalence-equation}  
		\frac{1}{C} \| v \|_{X} \leq \|v\|_{\tilde u} \le C \|v\|_X \qquad \text{for all $\tilde u \in U$, $v \in X$.}
	\end{equation}
	\item \label{M4} For all $u \in X \setminus \{0\}$ there is neighborhood $U \subset X \setminus \{0\}$ of $u$ and a constant $C=C(u)>0$ such that for all $u_1, u_2 \in U, \, v,w \in X$
	$$
	| \langle v,w \rangle_{u_1} - \langle v,w \rangle_{u_2} | \leq C \| u_1 - u_2 \|_X \|v\|_X \|w\|_X
	$$
\end{enumerate}
Now, let
$$
M \subset X \setminus \{0\}
$$
be a symmetric submanifold of class $C^{1,1}$, i.e. $M = -M$. In particular, for every $u \in M$, the tangent space $T_uM \subset X$ is a closed subspace of $X$, so it is a Hilbert space endowed with $\langle \cdot, \cdot \rangle_X$ and therefore also with any of the (equivalent) scalar products $\langle \cdot,\cdot \rangle_v$, $v \in X \setminus \{0\}$.

Then we can define a distance $d : M \times M \to \R$ induced by the metric $u \mapsto \langle \cdot, \cdot \rangle_u$ on $T M$, which is defined in the following standard way. For all $L>0$, $\gamma \in C^1([0,L],M)$ we define the length of the curve $\gamma$  as
$$ l(\gamma) := \int_0^{L} \| \gamma'(s) \|_{\gamma(s)} ds. $$
Then we define
\begin{equation}
	\label{eq:definition-metric-d}
	d(u,v) := \inf \{ l(\gamma) : \gamma \in C^1([0,1],M), \, \gamma(0) = u, \, \gamma(1)=v \} \qquad \text{for $u,v \in M$.}
\end{equation}
It is well known and easy to see that $d$ defines a metric on every connected component of $M$ \footnote{By definition, $d(u,v)= \infty$ if $u$ and $v$ are contained in different connected components of $M$}. Moreover, the relative topology inherited by $(X,\|\cdot\|_X)$ and the metric topology generated by $d$ are equivalent, see Lemma \ref{compatability-lipschitz} below.

In the following, we consider an even functional $\Phi \in C^{1,1}(M)$. Observe that \ref{M1} yields for all $u \in X$ the existence of gradients of $\Phi$ at the point $u$ with respect to the any of the (equivalent) scalar products $\langle \cdot, \cdot \rangle_X$ and $\langle \cdot, \cdot \rangle_v$, $v \in X \setminus \{0\}$. More precisely, for $u \in M$ and $v \in X \setminus \{0\}$, we let $\nabla_X \Phi(u) \in T_u M \subset X$ and $\nabla_v \Phi(u) \in T_u M \subset X$ be the unique vectors with the property 
$$
\Phi'(u) w = \langle \nabla_v \Phi(u),w \rangle_v = \langle \nabla_X \Phi(u),w \rangle_X \qquad \text{for all $w \in T_u M$.}
$$
Clearly, the assumption $\Phi \in C^{1,1}(M)$ implies that the map
$$
\nabla_X \Phi : M \to X, \qquad u \mapsto \nabla_X \Phi(u)
$$
is locally Lipschitz continuous, and, although less obvious, the $u$-dependent gradients have this property as well, see Lemma \ref{lipschitz-continuity-gradient} below.

In the following, we consider the critical sets
$$
K := \{ u \in M \::\: \Phi'(u) = 0\} \qquad \text{and}\qquad K_c:= \{ u \in K\::\: \Phi(u)=c\}\quad \text{for $c \in \R$.}
$$
Moreover, we let 
$$
\Phi^c:= \{u \in M\::\: \Phi(u) \le c\}, \qquad c \in \R
$$
denote the usual sublevel sets of $\Phi$. 

Now, we introduce the symmetric setting. Let $G$ be a topological group acting linearly on $X$ on the left. We denote this action by $g \ast u$ for $u \in X, g \in G$. We assume that the manifold $M \subset X$ and the functional $\Phi \in C^{1,1}(M)$ are $G$-invariant. In particular, this yields the implication
$$
v \in T_u M \quad \Longrightarrow \quad g \ast  v \in T_{g \ast  u}M \qquad \text{for $u \in M$, $g \in G$.}
$$
As mentioned in the beginning of this section, we do not assume that $\|\cdot\|_X$ is $G$-invariant, but we need the following key invariance condition on the family of scalar products $\langle \cdot, \cdot, \rangle_u$, $u \in X \setminus \{0\}$. 

\begin{enumerate}[resume,label=(M\arabic*),ref=(M\arabic*)]
	\item \label{M3} For all $u,v,w \in X \setminus \{0\}$ and $g \in G$ we have $\langle g \ast v, g \ast w \rangle_{g \ast u} = \langle v, w \rangle_{u} $
\end{enumerate}

As a consequence, we will obtain the $G$-equivariance of the $u$-dependent gradients in Lemma \ref{lemma:invariant}.

Now we introduce a new variant of the \textit{Cerami condition} that is better suited to our setting. We set
\begin{equation}
	\label{eq:def-N-w}
	N(w):= \inf \{\|u\|_u + d(u,w)\::\: u \in M\} \qquad \text{for $w \in M$.}
\end{equation}

\begin{definition}
	\label{N-bounded}  
	\begin{itemize}
		\item[(i)] A subset $A$ of $M$ is called $N$-bounded if
		\begin{equation}
			\label{eq:def-N-A}
			N(A):= \sup \limits_{w \in A}N(w)<\infty.    
		\end{equation}
		\item[(ii)] For $c \in \R$, we call a sequence $(u_n)_n$ an $N$-Cerami sequence at the level $c$ (in short: $(NC)_c$-sequence)
		if $\Phi(u_n) \to c$ and $\|\nabla_{u_n} \Phi(u_n)\|_{u_n} (1+ N(u_n)) \to 0$ as $n \to \infty$.
	\end{itemize}
\end{definition}

We now introduce the following key assumption depending on $c \in \R$, which we call the {\em $N$-Cerami condition relative to the group $G$ at level $c$ in $M$}.
\begin{enumerate}[label=$(NCG)_c$,ref=$(NCG)_c$]
	\item \label{NC} If $(u_n)_n \subset M$ is an $(NC)_c$-sequence, then, after passing to a subsequence, there exists a sequence $(g_n)_n$ in $G$ and $u \in M$ with $g_n  \ast  u_n \to u$ in $X$ as $n \to \infty$.
\end{enumerate}

Consider the neighborhoods\footnote{Here and in the following we define $\dist_M(v,A):= \inf \{d(v,u)\::\: u \in A\}$}
$$
A_{c,\rho}:= \{v \in M\::\: \dist_M(v,K_c)  < \rho\}
$$
One of our main results is the following deformation type lemma.
\begin{lemma}
	\label{deformation-lemma-cerami}
	Let $c \in \R$ be such that $\Phi^{-1}([c-2\varepsilon_0,c+2\varepsilon_0]) \subset M$ is complete with respect to the metric $d$ for some $\eps_0>0$, and suppose that $\Phi$ satisfies condition \ref{NC}. Then, for every $\rho>0$, there exists $\eps=\eps(c,\rho) \in (0,\eps_0)$ and a continuous function $\eta: [0,1] \times M \to M$ with the following properties:
	\begin{itemize}
		\item[(i)] $\eta(t,u)=u$ if $t=0$ or $u \not \in \Phi^{-1} ([c-2\eps,c+2\eps])$.
		\item[(ii)] $\eta(t,\cdot)$ is an odd homeomorphism of $X$ for all $t \in [0,1]$, i.e. $\eta(t,-u) = - \eta(t,u)$ for all $u \in X, t \in [0,1]$.
		\item[(iii)] $\eta(1,\Phi^{c+\eps} \setminus A_{c,\rho}) \subset \Phi^{c-\eps}$.
		\item[(iv)] $t \mapsto \Phi(\eta(t,u))$ is nonincreasing for all $u \in M$.  
		\item[(v)] $\eta$ is $G$-equivariant, i.e.,
		\begin{equation*}
			\label{eq:G-equivariance-flow-cerami}
			\eta(t,g \ast u)= g \ast \eta(t,u)\qquad  \text{for $u \in M$, $g \in G$, $t \in [0,1]$.}
		\end{equation*}  
	\end{itemize}  
      \end{lemma}

      \begin{remark}
        \label{remark-N-boundedness-Cerami}{\rm 
       (i) The relevance of this deformation lemma becomes clear in applications where the critical set $K_c$ and the neighborhoods $A_{c,\rho}$ are unbounded in the norm $\|\cdot\|_X$. In particular, this will be the case in the context of the translation invariant logarithmic Choquard equation where we have $\|g_n  \ast  u\|_X \to \infty$ for any $u \in X \setminus \{0\}$ with an appropriate sequence $(g_n)_n \subset G$, see Section~\ref{sec:choquard} below. Usually, deformation lemmata depending on Cerami type conditions require the boundedness of the associated critical set. In particular, we believe that a boundedness assumption is missing in the formulation of \cite[Lemma 2.6]{li-wang:11}. The introduction of the functional $N$ gives rise to a $G$-invariant notion of boundedness, and it readily follows from condition~\ref{NC} that $K_c$ and $A_{c,\rho}$ are $N$-bounded sets, see Lemma~\ref{N-boundedness-neighborhoods-critical} below.\\
       (ii) One might wonder why we have not simply considered the term $\|\nabla_{u_n} \Phi(u_n)\|_{u_n} (1+ \|u_n\|_{u_n})$ in place of $\|\nabla_{u_n} \Phi(u_n)\|_{u_n} (1+ N(u_n))$ in our definition of Cerami sequences, as it is also $G$-invariant. Since $N(u_n) \le  \|u_n\|_{u_n}$, this would give rise to a weaker variant of the Cerami condition. In fact we need the functional $N$ for the deformation arguments in the proof of Lemma~\ref{deformation-lemma-cerami}, as it is more closely related to the metric $d$ on $M$. }
      \end{remark}
      
We postpone the proof of Lemma~\ref{deformation-lemma-cerami} to Subsection \ref{subsec:cerami-case} and continue with the presentation of our $G$-invariant theory. Consider the Ljusternik-Schnirelmann values for $\Phi$ given by 
\begin{equation*}
	c_k := \inf\{c > 0 : \gamma(\Phi^c) \geq k\} \quad \in [-\infty,\infty] ,\qquad \text{for $k \in \N$.} 
\end{equation*}
where $\gamma$ denotes the \textit{Krasnoselskii genus}, see Subsection \ref{sec:compl-proof-theor} below. Before stating our main multiplicity result, we need further assumptions on the $G$-action on $X$.

\begin{enumerate}[label*=$(G_*)$]
	\item \label{Gstar} If $g_n \ast u \to v$ in $X$ for a sequence $(g_n)_n$ in $G$ and $u,v \in X$, then there exists $g \in G$ with $g \ast u =v$. 
\end{enumerate}
\begin{enumerate}[label*=$(I)_c$]
	\item \label{Ic}  $ g \ast  u \neq -u$ for all $u \in K_c$ and $g \in G$.
\end{enumerate}

Then, we have the following multiplicity result.

\begin{theorem}
	\label{c-k-to-infty}
	Suppose that $G$ satisfies \ref{Gstar}, and that there are values $c_*, c_\infty \in \R \cup \{+\infty\}$, $c_* < c_\infty$ with the following properties:
	\begin{itemize}
		\item[(i)] The metric space $(\Phi^{-1}([c_*,c]),d)$ is complete for all $c < c_\infty$.
		\item[(ii)] $k_*:= \lim \limits_{\stackrel{c \to c_*}{c>c_*}}\gamma(\Phi^c)< \infty$.
		\item[(iii)] $\lim \limits_{\stackrel{c \to c_\infty}{c<c_\infty}}\gamma(\Phi^c)=\infty$. 
		\item[(iv)] The conditions \ref{NC} and \ref{Ic} hold at every level $c \in [c_*,c_\infty)$.
	\end{itemize}
	Then we have
	$$
	c_* < c_{k_*+1} \le \dots \le c_{k} < c_\infty \quad \text{for $k > k_*$},\qquad \lim_{k \to \infty} c_{k} = c_\infty,
	$$
	and the values $c_k$, $k > k_*$ are critical values of the functional $\Phi$. Hence there exists a sequence of pairs of critical points $\{\pm u_k\}$ of $\Phi$ with $\Phi(u_k) \to c_\infty$ as $k \to \infty$.
\end{theorem}

In some situations arising in applications, the functional $\Phi$ does not satisfy the {\em Cerami condition relative to the group $G$}, but only a weaker variant related to a closed normal subgroup $G_0 \subset G$. In Subsection \ref{sec:variant-case} we will present results for this variant case.

\subsection{Some Lipschitz properties}
\label{sec:some-lipsch-prop}
We continue using the notation from the previous section. We need some local Lipschitz continuity results for functions defined on the submanifold $M \subset X$. For this it is convenient to check first that the metric $d$ on $M$ is locally equivalent to the metric induced by $\|\cdot\|_X$. More precisely, we have the following compatibility property.

\begin{lemma}
  \label{compatability-lipschitz}
  For every $u \in M$ the following holds.
  \begin{enumerate}
  \item[(i)] For every sequence $(u_n)_n \subset M$ we have 
  $\|u_n-u\|_X \to 0$ as $n \to \infty$ if and only if $d(u_n,u) \to 0$ as $n \to \infty$.  
  \item[(ii)] There exists a relative neighborhood $U \subset M$ of $u$ and $C>0$ with
  $$
  \frac{1}{C} \|v-w\|_X \le d(v,w) \le C \|v-w\|_X \qquad \text{for all $v,w \in U$.}
  $$
  \end{enumerate}
\end{lemma}

Note that, as a consequence of (i), the relative neighborhood in (ii) can be taken both with respect to the $X$-topology or with respect to the one induced by $d$. Therefore, local Lipschitz continuity properties of maps defined on $M$ with respect $\|\cdot\|_X$ and $d$ are equivalent.

\begin{proof}
We first show that there exists a relative neighborhood $U \subset M$ of $u$ with respect to the $X$-topology with the property that (ii) holds.  Since $M$ is a $C^1$-submanifold of $X$, there exist a closed subspace $X_* \subset X$, a relatively open neighborhood
  $U_0 \subset M$, an open ball $N_0= B_r(0) \subset X_*$ and a $C^1$-diffeomorphism $\psi: N_0 \to U_0$ (with respect to $\|\cdot\|_X$) with $\psi(0)=u$. Moreover, we may assume that the local equivalence property (\ref{local-uniform-equivalence-equation}) holds in $U_0$. 
  Now for every $v,w \in U_0$ we can choose the curve
  $$
  \gamma \in C^1([0,1],M) \quad \gamma(t):= \psi\bigl((1-t)\psi^{-1}(v) + t\psi^{-1}(w)\bigr)
  $$
  joining $v$ to $w$ to estimate that
  \begin{align*}
    d(v,w)  &\le  \int_0^1 \| \gamma'(s) \|_{\gamma(s)}ds \le C  \int_0^1 \| \gamma'(s) \|_{X} ds\\
    &= C \int_0^1 \| d\psi \bigl((1-t)\psi^{-1}(v) + t\psi^{-1}(w)\bigr)\psi^{-1}(w)-\psi^{-1}(v)\|_X ds\\
 &\le C \Bigl(\sup_{N_0} \|d\psi\|\Bigr)\|\psi^{-1}(v)-\psi^{-1}(w)\|_X \le \tilde C \|v-w\|_X.    
  \end{align*}
  Next, we put $N_1 := B_{\frac{r}{2}}(0)$ and $U:= \psi(N_1)$. By the Lipschitz continuity of $\psi$, we know that
  \begin{equation}
    \label{eq:U_0-upper-est}
  \|v-w\|_X \le L\|\psi^{-1}(v)-\psi^{-1}(w)\|_X \le Lr \qquad \text{for } v, w \in U_0,
  \end{equation}
   with a constant $L>0$ depending only on $U_0$. Moreover, for $v,w \in U$, we consider an arbitrary curve $\gamma \in C^1([0,1],M)$ joining the points $v$ and $w$. We then distinguish two cases. If $\gamma([0,1]) \subset U_0$, then $\eta: [0,1] \to N_0$, $\eta = \psi^{-1} \circ \gamma$ is a well-defined curve which joins the points $\psi^{-1}(v)$ and $\psi^{-1}(w)$, and we have $\eta'(t)= d\psi^{-1}(\gamma(t))\gamma'(t)$ and therefore 
$$
\int_0^1 \| \gamma'(s) \|_{\gamma(s)}ds \ge \frac{1}{C}  \int_0^1 \| \gamma'(s) \|_{X} ds \ge \tilde C \int_0^1 \|\eta'(s)\|_X ds \ge \tilde C \|\psi^{-1}(v)-\psi^{-1}(w)\|_X.
$$
If $\gamma([0,1]) \not \subset U_0$, we let $s_0:= \inf\{s \in [0,1]\::\: \gamma(s) \not \in U_0\}$. Then $\eta: [0,s_0) \to N_0$, $\eta = \psi^{-1} \circ \gamma$ is still well-defined, and
\begin{align*}
	\int_0^1 \| \gamma'(s) \|_{\gamma(s)}ds &\ge \int_0^{s_0} \| \gamma'(s) \|_{\gamma(s)}ds \ge \tilde C \int_0^{s_0} \|\eta'(s)\|_X ds\\
	 &\ge \tilde C \lim_{s \to s_0} \|\eta(s)- \eta(0)\|_X = \tilde C \lim_{s \to s_0} \|\eta(s)- \psi^{-1}(v)\|_X \ge \tilde C \frac{r}{2}. 
\end{align*}
In both case we may combine the obtained estimate with (\ref{eq:U_0-upper-est}) to get that
$$
\int_0^1 \| \gamma'(s) \|_{\gamma(s)}ds \ge C \|v-w\|_X .
$$
Taking the infimum over such paths, we get the desired lower bound
\begin{equation}\label{eq:lower_bound_d}
	d(v,w) \ge C \|v-w\|_X \qquad \text{for all $v,w \in U$.}
\end{equation}
We have thus shown that there exists a relative neighborhood $U \subset M$ of $u$ with respect to the $X$-topology with the property that (ii) holds. From this, it immediately follows that for every sequence $(u_n)_n \subset M$ the convergence 
$\|u_n-u\|_X \to 0$ as $n \to \infty$ implies the convergence $d(u_n,u) \to 0$ as $n \to \infty$. To complete the proof of the lemma, we now show the opposite implication. So assume that $d(u_n,u) \to 0$ as $n \to \infty$, and choose a relative neighborhood   $U \subset M$ of $u$ with respect to the $X$-topology with the property that (ii) holds. If $u_n \in U$ for all but finitely many $u_n$, then we conclude by (ii) that $\|u_n-u\|_X \to 0$, as claimed. So let us assume by contradiction that, after passing to a subsequence, we have $u_n \in M \setminus U$ for all $n \in \N$. Moreover, we may choose curves $\gamma_n \in C^1([0,1],M)$ with $\gamma_n(0) = u$, $\gamma_n(1)=u_n$ and the property that
\begin{equation}
  \label{eq:additional-est-1-equiv-dist}
l(\gamma_n) \le d(u_n,u) + \frac{1}{n} \to 0 \qquad \text{as $n \to \infty$.}
\end{equation}
We then choose $r>0$ with $U_*:=\overline{B_r(u)} \cap M \subset U$ and let $s_n:= \inf\{s \in [0,1]\::\: \gamma_n(s) \not \in U_*\} \in [0,1]$. Then we have curves $\eta_n: [0,s_n] \to U_*$, $\eta_n(t):= \gamma_n(t)$, and 
\begin{equation}
  \label{eq:additional-est-2-equiv-dist}
l(\gamma_n) \ge l(\eta_n) \ge d(u,\gamma_n(s_n)) \ge \frac{1}{C}\|u-\gamma_n(s_n)\|_X = \frac{r}{C}\qquad \text{for all $n \in \N$.}
\end{equation}
Here we used the fact that $\gamma_n(s_n) \in \partial B_r(u) \subset U$ by construction. The contradiction given by (\ref{eq:additional-est-1-equiv-dist}) and (\ref{eq:additional-est-2-equiv-dist}) finishes the proof.
\end{proof}

In the following, we let $\Pi_u \in \cL(X)$ denotes the $X$-orthogonal projection on $T_u M$ for $u \in M$. Since $M$ is of class $C^{1,1}$, the map
$$
\Pi: M \to \cL(X),\qquad u \mapsto \Pi_u,
$$
is locally Lipschitz continuous \footnote{Here the space $\cL(X)$ is endowed with the operator norm w.r.t. to $\|\cdot\|_X$, i.e., $\|T\|_{\cL(X)} := \sup \{ \|Tv \|\::\: v \in X,\,\|v\|_X \le 1\}$}. In the following we also need the local Lipschitz continuity of the map
$$
M \to \cL(X), \qquad u \mapsto P_u,
$$
where $P_u$ denotes the projection on $T_uM$ which is orthogonal with respect to the scalar product $\langle \cdot, \cdot \rangle_u$.
\begin{lemma}
  \label{lipschitz-continuity-u-projection}
For every $u \in M$ there exist a neighborhood $U \subset M$ of $u$ and $C>0$ with
  \begin{equation}
    \label{eq:P-projection-estimate}
  \|P_{u_1} - P_{u_2}\|_{\cL(X)} \le C \|u_1-u_2\|_X \qquad \text{for $u_1,u_2 \in U$.}
  \end{equation}
  
\end{lemma}

\begin{proof}
  In the following, the letter $C>0$ stands for a constant which may change its value in every step of the estimates.
  We fix a relative neighborhood $U \subset M$ of $u$ such that 
\begin{equation}
    \label{eq:Pi-projection-estimate}
  \|\Pi_{u_1} - \Pi_{u_2}\|_{\cL(X)} \le C \|u_1-u_2\|_X \qquad \text{for $u_1,u_2 \in U$}
  \end{equation}
  and such that the uniform equivalence in (\ref{local-uniform-equivalence-equation}) holds in $U$. We then let $u_1,u_2 \in U$. For simplicity, we write $P_i$ resp. $\Pi_i$ in place of $P_{u_i}$ and $\Pi_{u_i}$, $i=1,2$, in the following. We first note that
  \begin{equation}
    \label{eq:difference-decomposition}
  P_2 - P_1 = P_2 (\id - P_1) - (\id-P_2)P_1. 
  \end{equation}
  Moreover, for $w \in X$ we have 
  \begin{equation}
    \label{eq:C-1-2-first-est}
  \|(\id - P_2)P_1 w\|_{u_2} \le C_{1,2}  \|P_1 w\|_{u_1} \qquad \text{with}\quad  C_{1,2} := \sup_{v \in T_{u_1}M \setminus \{0\}} \frac{\|v-P_2 v\|_{u_2}}{\|v\|_{u_1}}.
  \end{equation}
  Since for $v \in T_{u_1}M \setminus \{0\}$ we have 
  $$
  \|v-P_2 v\|_{u_2} = \min\{ \|v-z\|_{u_2}\::\: z \in T_{u_2}M\} \le C \min\{ \|v-z\|_{X}\::\: z \in T_{u_2}M\} = C \|v-\Pi_2 v\|_X,
  $$
  it follows that 
\begin{align*}
  C_{1,2} &\le C \sup_{v \in T_{u_1}M \setminus \{0\}} \frac{\|v-\Pi_2 v\|_{X}}{\|v\|_{u_1}}\le
            C \sup_{v \in T_{u_1}M \setminus \{0\}} \frac{\|v-\Pi_2 v\|_{X}}{\|v\|_{X}}\\
  &\le C \sup_{v \in X \setminus \{0\}} \frac{\|\Pi_1 v-\Pi_2 v\|_{X}}{\|v\|_X}  = C \|\Pi_1-\Pi_2\|_{\cL(X)} \le C \|u_1-u_2\|_X,    
\end{align*}
where we used (\ref{eq:Pi-projection-estimate}) in the last step. Combining this with (\ref{eq:C-1-2-first-est}) gives 
\begin{equation}
  \label{eq:first-projection-est}
\|(\id - P_2)P_1 w\|_{u_2} \le C  \|u_1-u_2\|_X \|P_1 w\|_{u_1} \quad \text{for $w \in X$.}
\end{equation}
Moreover, for $w \in X$ we have, by \ref{M4}, 
\begin{align}
  &\|P_2 (\id - P_1) w\|_{u_2}^2 = \langle P_2(\id - P_1) w, (\id-P_1)w \rangle_{u_2}\nonumber\\
  &\le C \|u_1-u_2\|_X \|P_2(\id-P_1)w\|_X \|(\id-P_1)w\|_X + \langle P_2(\id - P_1) w, (\id-P_1)w \rangle_{u_1}\nonumber\\
  &\le C \|u_1-u_2\|_X \|P_2(\id-P_1)w\|_{u_2} \|(\id-P_1)w\|_{u_1}  + \langle P_2(\id - P_1) w, (\id-P_1)w \rangle_{u_1},  \label{eq:first-projection-est-1}
\end{align}
 where
 \begin{align}
 \langle P_2(\id - P_1) w, (\id-P_1)w \rangle_{u_1} &= \langle (\id-P_1) P_2(\id - P_1) w, (\id-P_1)w \rangle_{u_1} \nonumber\\
& \le \|(\id-P_1) P_2(\id - P_1) w\|_{u_1} \|(\id - P_1) w\|_{u_1} \label{eq:first-projection-est-2} 
 \end{align}
 and, by applying (\ref{eq:first-projection-est}) with $(\id - P_1) w$ in place of $w$ and the roles of $u_1$ and $u_2$ interchanged, 
 \begin{equation}
\label{eq:first-projection-est-3}   
 \|(\id-P_1) P_2(\id - P_1) w\|_{u_1} \le C \|u_1-u_2\|_{\cL(X)}\|P_2 (\id - P_1) w\|_{u_2}
 \end{equation}
Combining (\ref{eq:first-projection-est-1}), (\ref{eq:first-projection-est-2}) and (\ref{eq:first-projection-est-3}) gives
\begin{equation}
  \label{eq:second-projection-est}
 \|P_2 (\id - P_1) w\|_{u_2} \le C \|u_1-u_2\|_{X} \|(\id - P_1) w\|_{u_1}  
\end{equation}
 Combining this with (\ref{eq:difference-decomposition}) and (\ref{eq:first-projection-est}), it follows that 
 $$
\|(P_2-P_1)w\|_{u_2} \le C \|u_1-u_2\|_{X} \Bigl(\|P_1 w\|_{u_1}  +\|(\id - P_1) w\|_{u_1}\Bigr) \le C \|u_1 - u_2\|_{X}\|w\|_{u_1} 
$$
for $w \in X$. Invoking the uniform equivalence in (\ref{local-uniform-equivalence-equation}) for $u \in U$ once more, we conclude that 
$$
\|P_2-P_1 \|_{\cL(X)}\le C \|u_1-u_2\|_X \qquad \text{for $u_1,u_2 \in U$,}
$$
as claimed.
\end{proof}

From the previous result, we may deduce the following local Lipschitz continuity property for the $u$-dependent gradients.

\begin{lemma}
\label{lipschitz-continuity-gradient}
Let $\Phi \in C^{1,1}(M)$. Then the map $X \setminus \{0\} \to X$, $u \mapsto \nabla_u \Phi(u)$ is locally Lipschitz continuous.  
\end{lemma}

\begin{proof}
  Let $u \in M$. Since $\nabla_X \Phi$ is locally Lipschitz continuous, there exist an open (relative) neighborhood $U \subset M$ of $u_0$ and $C>0$ with the property that
  \begin{equation}
    \label{eq:lipschitz-continuity-gradient-eq-1}
  \|\nabla_X \Phi(u_1)-\nabla_X \Phi(u_2)\| \le C \|u_1-u_2\|_X \qquad \text{for $u_1,u_2 \in U$.}
  \end{equation}
  By Lemma~\ref{lipschitz-continuity-u-projection}, we may also assume that \eqref{eq:P-projection-estimate} and the uniform equivalence (\ref{local-uniform-equivalence-equation}) holds in $U$. Here 
and in the following, the letter $C>0$ stands for a constant depending only on $U$ which may change its value in every step of the estimates. For $w \in X$, we then have
  \begin{align}
\label{lipschitz-continuity-gradient-eq-1}
    \langle \nabla_{u_1} \Phi(u_1)&-\nabla_{u_2} \Phi(u_2),w \rangle_{u_1}\\
    &= \langle \nabla_{u_1} \Phi(u_1), w \rangle_{u_1} - \langle \nabla_{u_2} \Phi(u_2),w \rangle_{u_2}
    + \Bigl(\langle \nabla_{u_2} \Phi(u_2),w \rangle_{u_2}- \langle \nabla_{u_2} \Phi(u_2),w \rangle_{u_1}\Bigr),\nonumber
\end{align}
where
\begin{align}
  \langle \nabla_{u_1} \Phi(u_1), w \rangle_{u_1} &- \langle \nabla_{u_2} \Phi(u_2),w \rangle_{u_2}  = \Phi'(u_1)P_{u_1} w -\Phi'(u_2)P_{u_2} w \nonumber\\
  &= \langle \nabla_X \Phi(u_1),P_{u_1} w \rangle_X - \langle \nabla_X \Phi(u_2), P_{u_2} w \rangle_X\nonumber\\
                                                                          &=\langle \nabla_X \Phi(u_1) -\nabla_X \Phi(u_2),P_{u_1} w \rangle_X + \langle \nabla_X \Phi(u_2), P_{u_1} w - P_{u_2} w \rangle_X \nonumber\\
                                                                          &\le \|\nabla_X \Phi(u_1) -\nabla_X \Phi(u_2)\|_X \|w\|_X + \|\nabla_X \Phi(u_2)\| \|P_{u_1} w - P_{u_2} w \|_X \nonumber\\
                                                                          &\le C \|u_1-u_2\|_X \|w\|_X.                          \label{lipschitz-continuity-gradient-eq-2}
\end{align}
Here we used that we may assume, by making $U$ smaller if necessary, that $\|\nabla_X \Phi(u)\|$ remains bounded for $u \in U$. Moreover, by \ref{M4} we have  
$$
\langle \nabla_{u_2} \Phi(u_2),w \rangle_{u_1}- \langle \nabla_{u_2} \Phi(u_2),w \rangle_{u_2} \le C \|u_1-u_2\|_X \|\nabla_{u_2} \Phi(u_2)\|_X \|w\|_X \le C \|u_1-u_2\|_X \|w\|_X. 
$$
Combining the latter inequality with (\ref{lipschitz-continuity-gradient-eq-1}) and (\ref{lipschitz-continuity-gradient-eq-2}) we get 
$$
    \langle \nabla_{u_1} \Phi(u_1)-\nabla_{u_2} \Phi(u_2),w \rangle_{u_1} \le C \|u_1-u_2\|_X  \|w\|_X.
$$
Choosing $w = \nabla_{u_1} \Phi(u_1)-\nabla_{u_2} \Phi(u_2)$ gives
$$
\|\nabla_{u_1} \Phi(u_1)-\nabla_{u_2} \Phi(u_2)\|_{u_1}^2 \le C \|u-v\|_X \|\nabla_{u_1} \Phi(u_1)-\nabla_{u_2} \Phi(u_2)\|_X
$$
Since $\|\cdot\|_u$ and $\|\cdot\|_X$ are uniformly equivalent norms for $u \in U$, it follows that
$$
\|\nabla_{u_1} \Phi(u_1)-\nabla_{u_2} \Phi(u_2)\|_{X} \le C \|u_1-u_2\|_X \qquad \text{for $u_1,u_2 \in U$}
$$
as required.
\end{proof}

\begin{corollary}\label{corollary:loclipschitz}
	Let $\Phi \in C^{1,1}(M)$.  Then the negative gradient vector field 
	$$
        M \setminus K \to TM, \qquad 
        u \mapsto -\frac{\nabla_u \Phi(u)}{\| \nabla_u \Phi(u) \|_u^2}
        $$        
	is locally Lipschitz continuous.
\end{corollary}

\begin{proof}
  This follows in a standard way from Lemma~\ref{lipschitz-continuity-gradient}.
\end{proof}

Next we note tha $G$-invariance of the vector field $u \mapsto \nabla_u \Phi(u)$, which is a consequence of assumption \ref{M3}.

\begin{lemma}\label{lemma:invariant}
Let $\Phi \in C^{1,1}(M)$. Then the gradient vector field $u \mapsto \nabla_u \Phi(u)$ is $G$-equivariant. More precisely, we have
$$
\nabla_{g \ast u} \Phi(g \ast u) = g \ast \nabla_{u} \Phi(u) \quad \forall u \in M, g \in G.
$$
and
  \begin{equation}
    \label{eq:gradient-norm-G-invariance}
 \| \nabla_{g \ast u} \Phi(g \ast u) \|_{g \ast u} = \| \nabla_u \Phi(u) \|_u \quad \text{for all } u \in M,g \in G.
  \end{equation}
\end{lemma}
\begin{proof}
	Let $u \in M$ and $g \in G$. Since $\Phi$ is $G$-invariant, the chain rule gives $ \Phi'(g \ast u) g \ast v = \Phi'(u)v$
for every $v \in T_uM$ and therefore, by \ref{M3},
	$$\langle g^{-1} \ast \nabla_{g \ast u}\Phi(g \ast u),  v \rangle_{u} = \langle \nabla_{g \ast u}\Phi(g \ast u), g \ast v \rangle_{g \ast u} = \Phi'(g \ast u) g \ast v = \Phi'(u) v = \langle \nabla_u\Phi(u), v \rangle_u$$
for every $v \in T_uM$. We conclude that $ \nabla_{g \ast u} \Phi(g \ast u) = g \ast \nabla_u \Phi(u)$, and applying again \ref{M3} gives
	$$ \| \nabla_{g \ast u} \Phi(g \ast u) \|_{g \ast u}  = \| g \ast \nabla_{u} \Phi(u) \|_{g \ast u} = \| \nabla_u \Phi(u) \|_u.$$
\end{proof}

\subsection{A Deformation Lemma under the $(NCG)$-condition}
\label{subsec:cerami-case}
We continue using the notation from previous sections, so we let $M \subset  X \setminus \{0\}$
be a submanifold of class $C^{1,1}$ with $M = -M$. and we let $\Phi \in C^{1,1}(M)$ be an even functional. We also recall the function $N$ defined in \eqref{eq:def-N-w}
\begin{equation*}
N(w):= \inf \{\|u\|_u + d(u,w)\::\: u \in M\} \qquad \text{for $w \in M$.}
\end{equation*}
Chosing $w=u$ gives the basic inequality
\begin{equation}
  \label{N-w-basic-ineq-1}
N(w) \le \|w\|_w \qquad \text{for all $w \in M$.}
\end{equation}
Moreover, we have 
\begin{equation}
  \label{N-w-basic-ineq-2}
N(w) \le N(v) + d(v,w) \qquad \text{for $v,w \in M$.}
\end{equation}
We also recall the neighborhoods of the critical set $K_c$ of $\Phi$ at level $c \in \R$ given by  
$$
A_{c,\rho}:= \{v \in M\::\: \dist_M(v,K_c) < \rho\}, \qquad \rho>0.
$$
Observe that the sets $A_{c,\rho}$ are $G$-invariant and open with respect to the topology induced by $d$, hence also open with respect to the norm $\|\cdot\|_X$ by Lemma~\ref{compatability-lipschitz}.
\begin{lemma}
  \label{N-boundedness-neighborhoods-critical}
Let $c \in \R$ be chosen such that \ref{NC} holds. Then, for every $\rho >0$, the set 
$$
A_{c,\rho}= \{v \in M\::\: \dist_M(v,K_c) < \rho\}
$$
is $N$-bounded.
\end{lemma}

\begin{proof} 
  We first show that $N(K_c)<\infty$. Suppose by contradiction that there exists a sequence $(u_n)_n$ in $K_c$ with $N(u_n) \to + \infty$, which by (\ref{N-w-basic-ineq-1}) also implies that $\|u_n\|_{u_n} \to \infty$. Since $(u_n)_n$ is an $(NCS)_C$-sequence, condition \ref{NC} implies that there exist a sequence $(g_n)_n$ in $G$ and $u \in M$ with $g_n  \ast  u_n \to u$ in $X$ as $n \to \infty$. By the locally uniform equivalence assumption \ref{M1} and assumption \ref{M3}, it follows that
  $$
  \|u_n\|_{u_n} = \|g_n \ast  u_n\|_{g_n \ast  u_n} \le C \|g_n \ast  u_n\|_X \to C \|u\|_X,
  $$
  a contradiction. Hence $N(K_c)<\infty$. Now let $\rho >0$. For every $w \in A_{c,\rho}$ there exists $u \in K_c$ with $
	 d(u,w)< \rho $ and therefore, by (\ref{N-w-basic-ineq-2}), 
  $$
  N(w) \le N(u) + d(u,w) \le N(K_c) + \rho,
  $$
  which implies that $N(A_{c,\rho}) \le N(K_c) + \rho< \infty$.
\end{proof}

Now, we present the proof of Lemma~\ref{deformation-lemma-cerami}, our main deformation lemma.

\begin{proof}[Proof of Lemma~\ref{deformation-lemma-cerami}]
  Let $\rho>0$. First we find $\varepsilon \in (0,\varepsilon_0)$ with 
  \begin{equation}
    \label{eq:first-eps-delta-est-cerami}
  \|\nabla_u \Phi(u)\|_{u}(1+N(u)) \ge \frac{8\varepsilon}{\rho} (1+N(A_{c,\rho}))  \qquad \text{for $u \in \Phi^{-1} ([c-2\varepsilon,c+2\varepsilon]) \setminus A_{c,\frac{\rho}{4}}$,}
  \end{equation}
where $N(A_{c,\rho})$ is defined in (\ref{eq:def-N-A}) and finite by Lemma~\ref{N-boundedness-neighborhoods-critical}. Suppose that no such $\varepsilon$ exists. Then, for every $n \in \N$, there is $u_n \in \Phi^{-1} ([c-\frac{2}{n},c+\frac{2}{n}])\setminus A_{c,\frac{\rho}{4}}$ with
    $$
\|\nabla_{u_n} \Phi(u_n)\|_{u_n}(1+N(u_n)) < \frac{8}{\rho n}(1+N(A_{c,\rho})).
    $$
    Consequently, $(u_n)_n$ is a $(NCS)_c$-sequence for $\Phi$. Using condition \ref{NC}, we find $u \in K^c$ and a sequence $(g_n)_n$ in $G$ with $v_n:= g_n  \ast  u_n \to u$ in $X$ after passing to a subsequence. Then, we then have
    $$
    d(u_n, g_n^{-1} \ast u) = d(v_n,u) \to 0 \qquad \text{as $n \to \infty$,}
    $$
    whereas $g_n^{-1} \ast u \in K^c$ for all $n \in \N$. Consequently, $u_n \in A_{c,\frac{\rho}{4}}$ for sufficiently large $n$, which yields a contradiction. Hence (\ref{eq:first-eps-delta-est-cerami}) is true.

    Next, we consider the sets
	$$ 
A := \Phi^{-1}([c-2 \varepsilon, c+2\varepsilon]) \setminus A_{c, \frac{\rho}{4}}\quad \text{and} \quad B:= \Phi^{-1}([c- \varepsilon, c+\varepsilon])\setminus A_{c,\frac{\rho}{2}},
$$
and we define the map 
$$
\nu: M \to [0,1], \qquad \nu(u) := \frac{\dist_M(u,M \setminus A)}{\dist_M(u,M \setminus A) + \dist_M(u,B)}.
$$
We note the $\nu$ is locally Lipschitz continuous and satisfies $\nu(u) = 1$ if $u \in B$ and $\nu(u)=0$ if $u \in M \setminus A$.  Now, we define the locally Lipschitz continuous vector field 
\begin{equation}
\label{def-vector-field-W-cerami}  
W : M \to TM,\qquad 	W(u):= \begin{cases}
		-\nu(u) \|\nabla \Phi_u(u) \|_u^{-2} \nabla \Phi_u(u) \quad &\text{for } u \in A,\\
		0 \quad &\text{for } u \in M \setminus A.
	\end{cases}
\end{equation}
Since 
\begin{equation}
\label{W-u-est-cerami}
\|W(u)\|_u \leq  \frac{\rho}{8 \varepsilon (1+N(A_{c,\rho}))} (1+N(u)) \qquad  \text{for all $u \in M$}  
\end{equation}
holds by (\ref{eq:first-eps-delta-est-cerami}), it is not difficult to see that the flow generated by $W$ is defined globally,
i.e. there exists a map $\Theta : \R \times M \to M$ with
$$
\partial_t \Theta(t,u) = W(\Theta(t,u))\quad \text{and}\quad \Theta(0,u) = u \qquad \text{for all $u \in M$, $t \in \R$,}
$$
and with the property that $\Theta(t,\cdot): M \to M$ is a homeomorphism for all $t \in \R$. To see this, we observe that by (\ref{N-w-basic-ineq-2}) we have 
\begin{align*}
  N(\Theta(t,u))&\le N(u) + d(u,\Theta(t,u)) \le N(u)+ \int_0^t \|\partial_s \Theta(s,u)\|_{\Theta(s,u)}ds\\
&= N(u)+ \int_0^t \|W(\Theta(s,u)\|_{\Theta(s,u)}ds \le N(u) + \frac{\rho}{8 \varepsilon (1+N(A_{c,\rho}))}\Bigl(t + \int_0^t N(\Theta(s,u))\,ds\Bigr)   
\end{align*}
and therefore, by Gronwalls Lemma, the function $t \mapsto N(\Theta(t,u))$ remains bounded on finite time intervals, which by (\ref{W-u-est-cerami}) also implies that the function $t \mapsto \|W(\Theta(t,u))\|_u$ remains bounded on finite time intervals. Now, let $u \in M$ and let 
$$T_+ = T_+(u) := \sup \{ t > 0 : \eta(\cdot,u) \text{ is defined in } [0,t] \}  > 0.$$
First, observe that $u \in M \setminus A$ then $W(u) = 0$ so that $T_+ = \infty$. Thus, we take $u \in A$ and arguing by contradiction, we assume that $T_+ < \infty$. Now, we fix a positive sequence $t_n \nearrow T_+$ and set $u_n := \Theta(t_n,u)$. In this way, for $n < m$ we have
\begin{align}\label{eq:cauchy_d}
	d(u_n,u_m) \leq \int_{t_n}^{t_m} \| \partial_t \Theta(t,u) \|_{\eta(t,u)} dt \leq C |t_m - t_n|.
\end{align}
Now, since $(u_n)_n \subset \Phi^{-1}([c-2\varepsilon,c+2\varepsilon])$, from \eqref{eq:cauchy_d} we have $u_n \subset \Phi^{-1}([c-2\varepsilon_0,c+2\varepsilon_0])$ for large $n$. Then, since $\Phi^{-1}([c-2\varepsilon_0,c+2\varepsilon_0])$ is complete, we have $u_n \to u_+ \in \Phi^{-1}([c-2\varepsilon_0,c+2\varepsilon_0])$. Hence, we may extend the integral line of $u$ by $T_+(u_+)$ beyond $T_+(u)$, contradicting the definition of $T_+$.

Now, we define
$$
\eta: [0,1] \times M \to M, \qquad \eta(t,u) := \Theta(2 \varepsilon t,u).
$$
Then properties $(i)$ and $(ii)$ hold. To see $(iv)$, we compute
\begin{align}
	\frac{d}{dt} \Phi(\eta(t,u)) &= - \frac{2 \varepsilon}{\| \nabla_{\eta(t,u)} \Phi(\eta(t,u))\|_{\eta(t,u)}^2 } \langle \nabla_{\eta(t,u)} \Phi(\eta(t,u)) , \nu(\eta(t,u)) \nabla_{\eta(t,u)} \Phi(\eta(t,u)) \rangle_{\eta(t,u)} \nonumber \\
	& = - 2\varepsilon \nu(\eta(t,u)) \leq 0. \label{derivative-computation}
\end{align} 
To prove $(iii)$ we fix $u \in \Phi^{-1}([c-\varepsilon,c+\varepsilon]) \setminus A_{c,\rho}$. If there is $ t \in [0,1]$ such that $\Phi(\eta(t,u)) \leq c- \varepsilon$, then we get $(iii)$ from the monotonicity of $t \mapsto \Phi(\eta(t,u))$.
Otherwise, assume that $\Phi(\eta(t,u)) \in (c-\varepsilon,c+\varepsilon]$ for all $t \in [0,1]$, and suppose by contradiction that $\eta(t_*,u) \not \in B$ for some $t_* \in (0,1]$, which implies that $\eta(t_*,u) \in A_{c,\frac{\rho}{2}}$. Then, there exist $0\le t_1<t_2 \le t_*$ with
$$
d(\eta(t_1,u),\eta(t_2,u)) \ge \frac{\rho}{2} \qquad \text{and}\qquad \eta(s,u) \in A_{c,\rho} \setminus A_{c,\frac{\rho}{2}} \quad \text{for $t_1 < s < t_2$.}
$$
It follows, by (\ref{derivative-computation}) and (\ref{W-u-est-cerami}), that 
\begin{align*}
  \frac{\rho}{2} &\le d(\eta(t_1,u),\eta(t_2,u)) \leq \int_{t_1}^{t_2} \bigl\|\frac{d}{ds}\eta(s,u))\bigr\|_{\eta(s,u)}= 2 \eps \int_{t_1}^{t_2} \bigl\|W(\eta(s,u))\|_{\eta(s,u)}ds \\ &\le \frac{\rho}{4(1+N(A_{c,\rho}))}  \int_{t_1}^{t_2} (1+N(\eta(s,u)) ds
                                                                                                                                                                                        \leq \frac{\rho}{4}(t_2-t_1) \leq \frac{\rho}{4},
\end{align*}
a contradiction. Hence $\eta(t,u) \in B$ for $t \in (0,1)$, but then we have $\nu(\eta(t,u))=1$ for all $t \in (0,1)$ and therefore
\begin{align*}
	\Phi(\eta(1,u)) &= \int_0^1 \frac{d}{dt} \Phi(\eta(t,u)) dt + \Phi(u) = \Phi(u) -2\varepsilon \int_0^1 \nu(\eta( t,u)) dt  \\
	& \Phi(u)-2 \varepsilon \leq c - \eps,
\end{align*} 
which contradicts again our assumption. Thus $(iii)$ holds.

Finally, to see $(iv)$, we note that the sets $A$ and $B$ defined above are $G$-invariant, and therefore also the function $\nu$ is $G$-invariant. Combining this information with Lemma~\ref{lemma:invariant}, we infer that the vector field $W$ defined in (\ref{def-vector-field-W-cerami}) is $G$-equivariant, which in turn yields the $G$-equivariance of the maps $\Theta$ and $\eta$.

The same arguments apply with $\{\id, -\id\}$ in place of $G$  in the case where $M$ and $\Phi$ are invariant with respect to $-\id:X \to X$. Hence the function $\eta$ is odd.
\end{proof}

\subsection{Completion of the proof of Theorem~\ref{c-k-to-infty}}
\label{sec:compl-proof-theor}
We continue using the notation from previous sections, and we recall the notion of the \textit{Krasnoselskii genus}. Let $Z \subset X$ be a locally closed\footnote{A subset $Z \subset X$ is called \textit{locally closed in $X$} if there are subsets $A,B \subset X$, with $A$ open in $X$ and $B$ closed in $X$, such that $Z = A \cap B$.} symmetric subset of $X$. If $Z = \varnothing$, we put $\gamma(Z)=0$. If $Z\neq \varnothing$, we let $\gamma(Z)$ be the smallest integer $k \geq 1$ such that there exists an odd continuous function $h: Z \to \mathbb{S}^{k-1}$, where $\mathbb{S}^{k-1}$ is the unit sphere in $\mathbb{R}^{k}$. If no such $k$ exists, we define $\gamma(Z) := \infty$.

Observe that, since $M \subset X \setminus \{0\}$ is a submanifold, then it is locally closed in $X$, thus, the (relative) closed sets $\Phi^c \subset M$ are locally closed in $X$. We then recall the Ljusternik-Schnirelmann values for $\Phi$ given by 
\begin{equation*}
	c_k := \inf\{c > 0 : \gamma(\Phi^c) \geq k\} \quad \in [-\infty,\infty] ,\qquad \text{for $k \in \N$.} 
\end{equation*}
We then have the following.

\begin{proposition}
	\label{c-k-critical-value}
\begin{sloppypar}
		Let \ref{NC} be satisfied at $-\infty<c_k<\infty$ for some $k \in \N$ and that ${(\Phi^{-1}([c-\varepsilon_0,c+\varepsilon_0]),d)}$ is complete for some $\eps_0>0$. Then $c_k$ is a critical value of $\Phi$.
\end{sloppypar}
\end{proposition}

\begin{proof}
	Let $c:= c_k$. If $K_c = \varnothing$, then $A_{c,\rho}= \varnothing$ for every $\rho>0$, and by Lemma \ref{deformation-lemma-cerami} there exists $\eps=\eps(c,\rho) \in (0,\frac{\eps_0}{2})$ and an odd continuous function $\eta: M \to M$ with
	$\eta(\Phi^{c+\eps}) \subset \Phi^{c-\eps}$ and $\eta(u) = u$ if $u \notin \Phi^{-1}([c-2\varepsilon,c+2\varepsilon])$. Hence, the monotone property of the genus yields $\gamma(\Phi^{c+\eps}) \le \gamma(\Phi^{c-\eps})<k$, which contradicts the definition of $c_k$. It follows that $K_c \not = \varnothing$.
\end{proof}

Next, we derive the following important property from the abstract assumptions \ref{Gstar} and \ref{Ic} formulated in Subsection~\ref{sec:abstr-sett-main}. 

\begin{lemma}
	\label{krasnoselski-finite}
	Suppose that $G$ satisfies \ref{Gstar}, and let $c > c_*$ be such that the conditions \ref{NC} and \ref{Ic} hold. Then there exists $\rho_0=\rho_0(c)>0$ such that
	$$
	\gamma(A_{c,\rho}) < \infty,
	$$
	for all $\rho < \rho_0$.
\end{lemma}

\begin{proof}
	We consider the quotient space $\tilde M$ of $M$ with respect to the action of the group $G$. The elements of $\tilde M$ are the orbits $\bar u:= \{g  \ast  u\::\: g \in G \}$ with $u \in M$.\\
	{\bf Claim 1:} $\tilde M$ is a metric space with distance 
	$$
	d(\bar u,\bar v):= \inf \{d(u,v)\::\: u \in \bar u, v \in \bar v\} =  \inf \{ d(g  \ast  v,u)\::\: g \in G\} \quad \text{for arbitrary $u \in \bar u$, $v \in \bar v$.}
	$$
	Suppose that $d(\bar u,\bar v)=0$ for some $u,v \in M$.  Then there exist $g_n \in G$, $n \in \N$ with $\|g_n \ast  u- v \|_H \to 0$. If then follows from assumption \ref{Gstar} that there exists $g \in G$ with $g \ast  u =v$, hence $\bar u = \bar v$.
	To see the triangle inequality, let $u,v,w \in M$, and let $\eps>0$. Moreover, let $g,h \in G$ with
	$$
	d(g  \ast  v,u)<   d(\bar v,\bar u) + \frac{\eps}{2} \qquad \text{and}\qquad   d(h \ast  w,v)<   d(\bar w,\bar v) + \frac{\eps}{2}
	$$
	Then we have
	$$
	d((gh) \ast  w,u) \le d((g \ast (h  \ast  w),g \ast v)+d(g \ast v,u) = d(h \ast w,v)+ d(g \ast v,u) \le d(\bar v, \bar u)+ d(\bar w,\bar v) + \eps
	$$
	and therefore
	$$
	d(\bar w,\bar u) \le d(\bar v, \bar u)+ d(\bar w,\bar v) + \eps
	$$
	Since $\eps>0$ is arbitrary, it follows that $d(\bar w,\bar u) \le d(\bar v, \bar u)+ d(\bar w,\bar v)$, so the triangle inequality holds. This finishes the proof of Claim 1.
	
	Obviously, the metric defined on $\tilde M$ has the property that the projection map
	$$
	P_c: M \to \tilde M,\qquad u \mapsto P_c(u):= \bar u
	$$
	is continuous. More precisely, if $(u_n)_n \subset K_c$ is a sequence with $d(u_n,u) \to 0$ for some $u \in K_c$, then $d(\bar u_n,\bar u) \to 0$ as $n \to \infty$. 
	
	Next, we recall that $K_c$ is a $G$-invariant set, and we set
	$$
	\tilde K_c:= P_c(K_c) \subset \tilde M.
	$$
	{\bf Claim 2:} $\tilde K_c$ is a compact subset of $\tilde M$.
	
	Indeed, let $(\bar u_n)_n$ be a sequence in $\tilde K_c$, and choose $u_n \in \bar u_n$ for every $n$. Then $(u_n)_n$ is a $(NC)_c$-sequence. By \ref{NC}, there exists $u \in K_c$ and a sequence $(g_n)_n$ in $G$ with $g_n  \ast  u_n \to u$ in $X$, which by Lemma~\ref{compatability-lipschitz} implies that 
	$$
	d(\bar u,\bar u_n) \le d(g_n  \ast  u_n, u) \to 0 \qquad \text{as $n \to \infty$.}
	$$
	
	Next we note that, since, $-\id_X$ commutes with the $G$-action (so it is $G$-equivariant), it can be defined on $\tilde M$ by
	$$
	-\bar u = \overline{-u}
	$$
	Moreover, this mapping is continuous. Indeed, let $\bar u_n \to \bar u$ in $\tilde M$. Then there exists a sequence $(g_n)_n$ in $G$ with $g_n  \ast  u_n \to u$, and therefore also
	$$
	g_n  \ast  (-(u_n)) = -(g_n \ast u_n) \to -u,
	$$
	which implies that $\overline{-u_n} \to \overline{-u}$ in $\tilde M$. 
	
	{\bf Claim 3:} $\gamma(\tilde K_c) < \infty$.
	
	We argue as in \cite[P. 96]{s}. For this we first note that $-\id$ is fixed-point free on $\tilde K_c$ by assumption $(I)_c$. Since moreover $\tilde K_c$ is compact and $-\id$ is continuous on $\tilde K_c$,   there exists $\tau>0$ with
	\begin{equation}
		\label{eq:distance-property-tau}
		\text{$d(\bar u, -\bar u)>3 \tau$ for every $\bar u \in \tilde K_c$.}   
	\end{equation}
	Consider the sets
	$$
	L({\bar u}) := \{\hat v \in \tilde K_c\::\: d(\hat v,\bar u)< \tau\},\qquad K({\bar u}):=   L({\bar u}) \cup  L({\inv(\bar u)})
	$$ 
	for $\bar u \in \tilde K_c$, which are relatively open in $\tilde K_c$. Since $\tilde K_c$ is compact and
	$$
	\tilde K_c = \bigcup_{\bar u \in \tilde K_c}K(\bar u)
	$$
	there exists $\bar {u_1},\dots,\bar {u_1} \in \tilde K_c$ with 
	$$
	\tilde K_c = \bigcup_{i=1,\dots,m}K(\bar{u_i})= \bigcup_{i=1,\dots,m}\overline{K(\bar{u_i})}
	$$
	where the closure is taken with respect to the metric $d$. Since, as a consequence of
	(\ref{eq:distance-property-tau}), we have $\gamma(\overline{K(\bar{u_i})})\le 1$ for every $i=1,\dots,m$. it follows by the subadditivity property of the genus that $\gamma(\tilde K_c)<0$, as claimed.
	
	We now complete the proof of the lemma. Since $\gamma(\tilde K_c)<0$, there exists $k \in \N$ and a continuous map
	$$
	\tilde h: \tilde K_c \to S^{k-1}
	$$
	with $\tilde h(- \bar u)= -\tilde h(\bar u)$ for all $\bar u \in \tilde K_c$. By Tietze's extension theorem, there exists a continuous extension
	$$
	h_*: \tilde M  \to \R^k
	$$
	of $h$. Replacing $h_*$ by the function $u \mapsto \frac{1}{2}\Bigl(h_*(u)-h_*(-u)\Bigr)$, we may also assume that
	$h(-u)= -h(u)$ for all $u \in \tilde M$. 
	
	Next, let $U:= h_*^{-1}(\R^k \setminus \{0\})$. Then $U$ is an open neighborhood of $\tilde K_c$. Since $\tilde K_c$ is compact, there exists $\rho_0>0$ with 
	$$
	N_\rho:= \{\bar u \in \tilde M\::\: \inf_{\bar v \in \tilde K_c}d(\bar u,\bar v)\le 2\rho \} \subset U \qquad \text{for all } \rho < \rho_0.
	$$
	Note that, by definition of the metric $d$, we have 
	$$
	P_c(A_{c,\rho}) \subset N_\rho,
	$$
	hence we may define a continuous map $h: A_{c,\rho} \to \R^k \setminus \{0\}$ by $h(u)=P_c(h_*(u))$. This map also has the property that $h(-u)=-h(u)$ for all $u \in A_{c,\rho}$. This implies that $\gamma(A_{c,\rho})< \infty$, as claimed. 
\end{proof}
Now, we prove our main multiplicity result

\begin{proof}[Proof of Theorem \ref{c-k-to-infty}]
	By assumptions (i) and (iii), we have $c_{k_*+\ell} \in (c_*,c_\infty)$ for every $\ell \in \N$. Moreover, every value $c_{k_*+\ell},$ $\ell \in \N$ is a critical point of $\Phi$ as a consequence of Proposition~\ref{c-k-critical-value}. To complete the proof, it thus remains to show that
	$$
	c := \lim_{\ell \to \infty}c_\ell = c_\infty.
	$$
	So suppose by contradiction that $c \in (c_*,c_\infty)$. According to Proposition~\ref{krasnoselski-finite}, there exists $\rho>0$ with $\gamma(A_{c,\rho})< \infty$. We may then apply the deformation lemma at $c$ to obtain $\eps>0$ such that $c - 2 \varepsilon > c_*$ and an odd continuous map $\eta:M \to M$ such that $\eta(\Phi^{c + \varepsilon} \setminus A_{c,\rho}) \subset \Phi^{c - \varepsilon}$ with $\eta(u) = u$ for $u \notin \Phi([c -2\varepsilon, c + 2 \varepsilon])$. Observe that $\gamma(\Phi^{c - \varepsilon}) < \infty$, since $c - \varepsilon < c_k$, for $k$ large enough. We now use the subadditivity and monotonicity of the genus to see that
	$$
	\gamma(\Phi^{c + \varepsilon}) \leq \gamma((\Phi^{c + \varepsilon} \setminus A_{c,\rho})) + \gamma(A_{c,\rho}) \leq \gamma(\Phi^{c - \varepsilon} ) + \gamma(A_{c,\rho}) < \infty.
	$$
	This yields a contradiction, since $c_k \leq c$ for all $k \in \N$.
\end{proof}

\subsection{A variant case}\label{sec:variant-case}
In this section we let $X$, $M$, $G$, and $\Phi$ be given as in Subsection~\ref{sec:abstr-sett-main}. We fix such a closed subgroup, and we let $X_0 \subset X$ resp. $M_0$ denote the subsets of $G_0$-invariant functions in $X$, $M$, respectively. Here and in the following, an element $u \in X$ is called $G_0$-invariant if $g \ast  u = u$ for every $g \in G_0$. It is clear that $X_0 \subset X$ are closed subspaces and therefore Hilbert spaces with norms $\|\cdot\|_H$ and $\|\cdot\|_X$. Note also that, since $G_0$ is a \text{normal} subgroup of $G$, we have
\begin{equation}
  \label{eq:G-0-invariance-X_0}
h \ast (g \ast  u) =  g \ast  u \qquad \text{for every $g \in G$, $h \in G_0$ and every $u \in X_0$,}
\end{equation}
so the action of $G$ maps $G_0$-invariant elements to $G_0$-invariant elements.

We may then apply the abstract theory of the previous section with $X$ and $M$ replaced by $X_0$ and $M_0$ and $\Phi$ replaced by the restriction of $\Phi$ to $M_0$, which we again denote by $\Phi$. 

In the following, for $c \in \R$, we redefine the sets $\Phi^c$ and $K_c$ by setting
$$
\Phi^c:= \{u \in M_0\: : \: \Phi(u) \le c\}
$$
and 
$$
K_c:= \{ u \in M_0\::\: \Phi'(u)=0,\: \Phi(u)=c\}.
$$
With these new definitions, we then consider the Ljusternik-Schnirelmann values as before by setting  
\begin{equation}\label{eq:ak}
c_k := \inf\{c > 0 : \gamma(\Phi^c) \geq k\} \quad \in [-\infty,\infty] ,\qquad \text{for $k \in \N$.} 
\end{equation}
Moreover, we redefine the sets $A_{c,\rho}$ for $\rho>0$ by setting 
$$
A_{c,\rho}:= \{v \in M_0\::\: d(u,v) < \rho \: \text{for some $u \in K_c$}\}.
$$
Moreover, we make use of the following conditions which are natural weaker replacements of the assumptions \ref{NC} and \ref{Ic} for the present context.

\begin{enumerate}[label*=$(NCG_0)_c$]
\item \label{NCG0} If $(u_n)_n \subset M_0$ is a sequence with
$\Phi(u_n) \to c$ and $\|\Phi'(u_n)\|_{u_n}(1+N(u_n))\to 0$ as $n \to \infty$, then, after passing to a subsequence, there exists a sequence $(g_n)_n$ in $G$ and $u \in X_0$ with $g_n  \ast  u_n \to u$ in $X$ as $n \to \infty$.
\end{enumerate}

\begin{enumerate}[label*=$(I)_c'$]
\item\label{Ic'} For all $u \in K_c$ and $g \in G \setminus G_0$ we have $g \ast  u \neq -u$ .     
\end{enumerate}

Note that \ref{Ic'} could be equivalently formulated by requiring that $g \ast  u \neq -u$ for all $u \in K_c$ and $g \in G$, since for $u \in K_c$ we have $g \ast  u =u$ for all $g \in G_0$ and therefore 
$$
g  \ast  u \neq -u \quad \text{for all $u \in K_c$, $g \in G_0$} \quad \text{iff}\quad u \neq 0,
$$
which holds, since we assume $M \subset X \setminus\{0\}$. We prefer to state \ref{Ic'} as above to emphasize the role of the subgroup $G_0$. An application of Theorem~\ref{c-k-to-infty} with the replacements above then yields the following statement. 

\begin{theorem}
  \label{c-k-to-infty-variant}
	Suppose that $G$ satisfies \ref{Gstar}, and that there are values $c_*, c_\infty \in \R \cup \{+\infty\}$, $c_* < c_\infty$ with the following properties:
\begin{itemize}
		\item[(i)] The metric space $(\Phi^{-1}([c_*,c]),d)$ is complete for all $c < c_\infty$.
		\item[(ii)] $k_*:= \lim \limits_{\stackrel{c \to c_*}{c>c_*}}\gamma(\Phi^c)< \infty$.
		\item[(iii)] $\lim \limits_{\stackrel{c \to c_\infty}{c<c_\infty}}\gamma(\Phi^c)=\infty$. 
	\item[(iv)] The conditions \ref{NCG0} and \ref{Ic'} hold at every level $c \in (c_*,c_\infty)$.
\end{itemize}
Then we have
$$
c_* < c_{k_*+1} \le c_{k_*+2} \le c_{k*+\ell} \to c_\infty \quad \text{as $\ell \to \infty$,}
$$
and these values are critical points of the functional $\Phi$. Hence there exists a sequence of pairs of critical points $\{\pm u_k\}$ of $\Phi$ with $\Phi(u_k) \to c_\infty$ as $k \to \infty$.
\end{theorem}

\section{A $G$-equivariant functional-analytic setting for the logarithmic Choquard equation}
\label{sec:preliminaries}
In the following sections, we wish to apply the abstract theory developed in Section \ref{sec:LS theory} to obtain solutions to the Choquard equation (\ref{eq:choquard_equation}). In the following, $|u|_p$ stands for the
$L^p$-norm of a function $u \in L^p(\R^2)$, $1 \le p \le \infty$. Let $H:= H^{1}(\mathbb{R}^{2})$ be the usual Sobolev space endowed with
the standard inner product
\begin{equation*}
    \langle u, v\rangle_H =\int_{\mathbb{R}^{2}} \left(\nabla u \cdot \nabla v
     + uv\right)dx   \quad \text{for}\ u, v \in H^{1}(\mathbb{R}^{2}),
\end{equation*}
and the induced norm denoted by $\|u\|_H=\langle u, u\rangle_H^{1/2}$.
Moreover, for any measurable function
$u: \mathbb{R}^{2}\rightarrow \mathbb{R}$, we define
\begin{equation*}
    |u|_{*} := \left( \int_{\mathbb{R}^{2}} \log\left(1+|x|\right)u^{2}
     dx\right)^{1/2}\in [0, \infty].
\end{equation*}
We also define the Hilbert space
\begin{equation}
\label{definition-X}
    X:=\left\{u \in H^{1}(\mathbb{R}^{2}): |u|_{*}< \infty\right\}
\end{equation}
\vskip -0.1 true cm\noindent
with the scalar product given by
$$ \langle u,v \rangle_X := \langle u,v \rangle_H + \langle u,v\rangle_* \qquad \langle u,v\rangle_* := \int_{\R^2} \log(1+|x|)u(x)v(x) dx ,$$
and the induced norm $\|u\|_{X}:= \sqrt{\|u\|^{2}+|u|^{2}_{*}}$. 

As mentioned in the introduction, the space $(X,\|\cdot\|_X)$ can be used to set up a variational formulation of the logarithmic Choquard equation~(\ref{eq:choquard_equation}), but the norm of $X$ is not translation invariant. In this section, we wish to construct a family of scalar products $\langle \cdot,\cdot \rangle_u$, $u \in X$ which satisfies the abstract assumptions of Section~\ref{sec:LS theory}. 
The main tool for this is a generalized barycenter map as constructed in \cite{cp,bw}. We now recall the construction of a generalized barycenter and show it gives rise to a Lipschitz continuous map.   

\subsection{A Lipschitz continuous generalized barycenter map}\label{subsec:barycenter}

Let $E(N)$ denote the group of euclidean motions on $\R^N$, i.e., the transformation group generated by rotations and translations in $\R^N$. For $p \ge 1$, the standard orthogonal action of $E(N)$ on $L^p(\R^N)$ is given by
$$
[g \ast  u](x) = u (g^{-1}x) \qquad \text{for every $u \in L^p(\R^N)$, $g \in E(N)$ and a.e. $x \in \R^N$.}
$$
We call a map \( \beta : L^p(\mathbb{R}^N) \setminus \{0\} \to \mathbb{R}^N \) a generalized barycenter map on \( L^p(\mathbb{R}^N) \) if it is continuous and equivariant with respect to this action of $E(N)$, i.e.,  we require that
\begin{equation}
  \label{eq:defining-prop-barycenter}
\beta(g  \ast  u) = g(\beta(u)) \ \text{for \( g \in E(N) \), \( u \in L^p(\mathbb{R}^N) \setminus \{0\} \)}.
\end{equation}
Maps with these properties have been constructed in \cite{bw,cp}. Property \eqref{eq:defining-prop-barycenter} yields
\[
\beta(u(\cdot - b)) = \beta(u) + b, \quad \text{and} \quad \beta(u \circ A^{-1}) = A \beta(u) \qquad \text{for $b \in \mathbb{R}^N$ and $A \in O(N)$.} 
\]
As a consequence, an even function, especially a radial function \( u \in L^p(\mathbb{R}^N) \setminus \{0\} \), has barycenter \( \beta(u) = 0 \). More generally, if \( u \in L^p(\mathbb{R}^N) \setminus \{0\} \) is invariant with respect to a subgroup \( G \subset O(N) \), then \( \beta(u) \in (\mathbb{R}^N)^G = \{ x \in \mathbb{R}^N \mid gx = x \, \text{for all } g \in G \} \).

We will prove now that the barycenter map constructed in \cite[Theorem 2.1]{bw} is locally Lipschitz continuous.

\begin{theorem}\label{thm:barycenter}
	For any \( p \in [1, \infty) \), there exists a generalized barycenter map \( \beta = \beta_p : L^p(\mathbb{R}^N) \setminus \{0\} \to \mathbb{R}^N \) which is locally Lipschitz continuous and satisfies \( \beta(|u|) = \beta(u) \).
\end{theorem}
\begin{proof}
	For all $u \in L^p(\R^N)$ we define
	$$ \hat{u} : \R^N \to \R \quad \hat{u}(x) := \int_{B_1(x)} |u(y)|^p dy. $$
	Now, we fix $p \in [1,\infty)$ and $u \in L^p(\R^N)$ and observe that 
	\begin{align*}
		|\hat{u}(x) - \hat{v}(x)| &\leq \int_{B_1(x)} \left| |u|^p - |v|^p \right| \leq p \int_{B_1(x)} |u-v|(|u|+|v|)^{p-1} \\
		&\leq p \left( \int_{B_1(x)}|u-v|^p \right)^{1/p} \left(\int_{B_1(x)}(|u| + |v|)^p  \right)^{{(p-1)}/p }\\
		&\leq p2^{p-1}|u-v|_p(|u|_p^p + |v|_p^p)^{(p-1)/p} \leq p 2^{p-1}3^{(p-1)/p}|u|_p^{p-1}|u-v|_p,
	\end{align*}
	for all $v \in B_\delta(u) \subset L^p(\R^N)$, where $\delta> 0$ is chosen such that $|v|_p^p < 2|u|_p^p$ in $B_\delta(u)$. Thus
	\begin{equation}\label{eq:infinitypbound}
		|\hat{u} - \hat{v}|_\infty \leq  p 2^{p-1}3^{(p-1)/p}|u|_p^{p-1}|u-v|_p =: c(u)|u-v|_p
	\end{equation}
	for all $v \in B_\delta(u)$.
	
	For \( v \in L^p(\R^N) \), we consider the set
	\[
	\Omega(v) = \left\{ x \in \mathbb{R}^N : \hat{v}(x) > \frac{\|\hat{v}\|_\infty}{2} \right\}.
	\]
	If \( v \neq 0 \), then \( \Omega(v) \subset \mathbb{R}^N \) is compact and has nonempty interior. Moreover,
	\[
	\beta_1(v) := \int_{\Omega(v)} \left( \hat{v}(x) - \frac{\|\hat{v}\|_\infty}{2} \right) dx > 0 \quad \text{for every } v \in L^p(\R^N) \setminus \{0\}
	\]
	Therefore, setting
	\[
	\beta_0(v) := \int_{\Omega(v)} x \left( \hat{v}(x) - \frac{\|\hat{v}\|_\infty}{2} \right) dx \in \mathbb{R}^N,
	\]
	the function 
	\[ \beta : L^p(\R^N) \setminus \{0\} \to \mathbb{R}^N \quad \beta(v) = \frac{\beta_0(v)}{\beta_1(v)}, \] is well defined and satisfies $\beta(|u|) = \beta(u)$.
	
	Now, we prove that $\beta_0$ is locally Lipschitz. We set
	$$ K := \left\{ x \in \R^N : \hat{u}(x) \geq \frac{|\hat{u}|_\infty}{4} \right\}.$$
	Observe that $K$ is a compact set containing $\Omega(u)$. Moreover, from the bound \eqref{eq:infinitypbound}, we can adjust $\delta$ so that
	$$ \Omega(v) \subset K \quad \text{for all } v \in B_\delta(u). $$
	
	It follows that
	\begin{align*} 
		|\beta_0(u) - \beta_0(v)|_2 &\leq \int_K \left| \left( x \left( \hat{u}(x) - \frac{\|\hat{u}\|_\infty}{2} \right)  \right)^+ - \left( x \left( \hat{v}(x) - \frac{\|\hat{v}\|_\infty}{2} \right)  \right)^+ \right| \\
		&\leq \int_K |x(\hat{u}(x) - \hat{v}(x))| + \int_K |x| \left| \frac{|\hat{u}|_\infty}{2} - \frac{| \hat{v}|_\infty}{2} \right| \\
		&\leq \frac{3M|K|}{2}  |\hat{u} - \hat{v}|_\infty  \leq \frac{3M|K|}{2}c(u) |u-v|_p \qquad \text{for all } u \in B_\delta(u),
	\end{align*}
	where $M$ is a bound for the compact set $K$. In a similar way, we can show that $\beta_1$ is locally Lipschitz. Then, since $\beta_1$ is continuous and $\beta_1(u) > 0$, it is locally bounded away from $0$. From this it follows that $\beta = \beta_0/\beta_1$ is locally Lipschitz continuous.
\end{proof}

\begin{remark}
  \label{homogeneity-barycenter}
  An inspection of the construction above shows the homogeneity property
  $$
  \beta(tu) = \beta(u)= \beta(|u|) \qquad \text{for all $u \in L^p(\R^N) \setminus \{0\}$, $t \in \R \setminus \{0\}$.}
  $$
\end{remark}

\subsection{An equivariant metric}\label{subsec:metric}

With the aid of the generalized barycenter map $\beta : L^2(\mathbb{R}^2) \setminus \{0\} \to \mathbb{R}^2$
discussed in the previous section, we now define a family of scalar products $\langle \cdot, \cdot \rangle_{u}$, $u \in L^2(\R^2) \setminus \{0\}$. As before, we let $E(2)$ denote the group of euclidean motions in $\R^2$, and we note that every element $g \in E(2)$ can be written in the form $g=g_{A,b}$ with 
$$
g_{A,b}x = A x+b \qquad \text{for some $A \in O(2)$, $b \in \R^2$.}
$$
If the context allows it, we will write $g_{b} = g_{id,b}$ and $g_{A} = g_{A,0}$ for short. Clearly, $b \in \R^2$ is uniquely given by $b = g(0)$, and $A$ is uniquely given by $Ax = g(x)-b$. We also note that
$g_{A,b}^{-1}y = A^{-1}(y-b)$ for $y \in \R^2$, and hence
$$
g_{A,b}^{-1} = g_{A^{-1},-A^{-1}b}
$$
With this notation, we now define\footnote{We deliberately defined $\langle \cdot,\cdot \rangle_u$ for general $u \in L^2(\R^2) \setminus \{0\}$ and not only for $u \in X \setminus \{0\}$. The continuous dependence of these scalar products with respect to $|u|_2$ will turn out to be useful.}  
\begin{equation}
  \label{eq:defition-g-norm}
\langle v,w \rangle_u := \langle g_{\beta(u)}^{-1} \ast v , g_{\beta(u)}^{-1} \ast w \rangle_X = \langle v(\cdot + \beta(u)) ,  w(\cdot + \beta(u)) \rangle_X  
\end{equation}
for $u \in L^2(\R^2) \setminus \{0\}$, $v,w \in X$. Our aim is to show that the family of scalar products $\langle \cdot,\cdot \rangle_u$, $u \in X$ satisfies the abstract assumptions in Section~\ref{sec:LS theory}. We first collect some basic observations. First, we note that
\begin{equation}
  \label{eq:homogeneity-scalar-prduct-family}
\langle \cdot, \cdot \rangle_{tu} = \langle \cdot, \cdot \rangle_{u}= \langle \cdot, \cdot \rangle_{|u|} \qquad \text{for all $u \in L^2(\R^N) \setminus \{0\}$, $t \in \R \setminus \{0\}$}
\end{equation}
by Remark~\ref{homogeneity-barycenter}. Moreover, we note that for the family of induced norms $\|\cdot\|_u$, $u \in X$ we have the uniform lower bounds
\begin{equation}
  \label{eq:uniform-lower-bounds}
\|v\|_u \ge \|v\|_H \ge |v|_2 \qquad \text{for all $u \in L^2(\R^2)$, $v \in X$,}  
\end{equation}
which follow from (\ref{eq:defition-g-norm}) and the definition of $\|\cdot\|_X$. Next we note the following.

\begin{lemma}
\label{M3-satisfied-lemma-0}
We have 
\begin{equation}
\label{M3-satisfied-lemma-sufficient}
\langle g \ast v, g \ast w \rangle_{g \ast u} = \langle v, w \rangle_u \quad \text{for all $g \in E(2)$, $v,w \in X$ and $u \in L^2(\R^2) \setminus \{0\}$.}
\end{equation}
\end{lemma}

\begin{proof}
Let $g=g_{A,b} \in E(2)$ with $A \in O(2)$, $b \in \R^2$. Moreover, let $v,w \in X$ and $u \in X \setminus \{0\}$. With
  \begin{equation}
    \label{eq:beta-beta-prime-relation}
\beta := \beta(u)\quad \text{and}\quad \beta':= \beta(g \ast u)= g(\beta(u))= g(\beta),
  \end{equation}
we then have 
$$
\langle v, w \rangle_u = \langle g_{\beta}^{-1} \ast v, g_{\beta}^{-1} \ast w \rangle_{X}= \langle v(\cdot + \beta), w(\cdot + \beta) \rangle
$$
and
\begin{align*}
  \langle g \ast v, g \ast w \rangle_{g \ast u} &= \langle (g_{\beta'}^{-1} g) \ast v, (g_{\beta'}^{-1} g) \ast w \rangle_{X}=\langle v (g^{-1}(g_{\beta'}(\cdot))),wg^{-1}(g_{\beta'}(\cdot)))\rangle_{X}  \\
                                                &= \langle v (g^{-1}(\cdot + \beta'),w(g^{-1}(\cdot + \beta')\rangle_{X}= \langle v (A^{-1}(\cdot + \beta'-b),w(A^{-1}(\cdot + \beta'-b)\rangle_{X}\\
                                                &= \langle v (A^{-1}(\cdot) +g^{-1}(\beta')),w(A^{-1}(\cdot) +g^{-1}(\beta'))\rangle_{X}\\
                                                &= \langle v (\cdot +g^{-1}(\beta')),w(\cdot +g^{-1}(\beta'))\rangle_{X},
\end{align*}
where we used the rotational invariance (i.e., the $O(N)$-invariance) of $\langle \cdot, \cdot \rangle_X$ in the last step.
Since $g^{-1}(\beta') = \beta$ by (\ref{eq:beta-beta-prime-relation}), we thus obtain (\ref{M3-satisfied-lemma-sufficient}), as required. 
\end{proof}

The next lemma implies that the family $\langle \cdot,\cdot \rangle_u$, $u \in X \setminus \{0\}$ satisfies \ref{M1} and \ref{M4}.

\begin{lemma}\label{lemma:scalarproduct}
  For every $u \in L^2(\R^2) \setminus \{0\}$, there exist a neighborhood $U \subset L^2(\R^2)$ of $u$ and a constant $C>0$ such that for all $u_1, u_2 \in U$ and $v,w\in X$ we have
  \begin{equation}
    \label{eq:M1-check}
\frac{1}{C} \| v \|_{X} \leq \|v\|_{u_1} \le C \|v\|_X 
\end{equation}
and
  \begin{equation}
    \label{eq:M4-check}
    \left|  \langle v,w \rangle_{u_1} -  \langle v,w \rangle_{u_2} \right| \leq C\|u_1 - u_2\|_X \int_{\R^2}|v(y)w(y)|dy.
  \end{equation}  
\end{lemma}

\begin{proof}
  We first note that $\log r - \log s \le |r-s|$ for $r,s \ge 1$. Moreover, we have
    $$
        \int_{\R^2} \log(1 + |y|) v(y + \beta(u_i)) w(y + \beta(u_i)) dy = \int_{\R^2} \log(1 + |y - \beta(u_i)|) v(y) w(y) dy\qquad \text{for $i=1,2$}  
	$$
        and thus
        \begin{align*}
          &\left|  \langle v,w \rangle_{u_1} -  \langle v,w \rangle_{u_2} \right|\\
          &\le \int_{\R^2} \Bigl|\log(1 + |y - \beta(u_1)|)-\log(1 + |y - \beta(u_2)|\Bigr||v(y)w(y)|dy\\
          &\le  \int_{\R^2} \Bigl| |y - \beta(u_1)|-|y - \beta(u_2)|\Bigr| |v(y)w(y)|dy\\
          &\le  |\beta(u_1)-\beta(u_2)| \int_{\R^2}|v(y)w(y)|dy
        \end{align*}
        Hence, since $\beta$ is locally Lipschitz continuous in $L^2(\R^2)$, we find a neighborhood $U \subset X \setminus \{0\}$, open in the relative $L^2(\R^2)$-topology, and a constant $C>0$ with the property that (\ref{eq:M4-check}) holds for all $u_1, u_2 \in U$ and $v,w\in X$. From this and the Cauchy-Schwarz inequality, we obtain \ref{M4}.

        Next, let $\kappa>0$. It is easy to see that there exists $C = C(\kappa)>1$ with 
        $$
        \frac{\log r}{\log s} \le C \qquad \text{for $s \ge 2,\: |r-s| \le \kappa$.}
        $$
        Hence for every $\beta \in \R^2$ with $|\beta| \le \kappa$ we have
        $$
        \log(1 + |y - \beta| \le C \log(1 + |y|) \qquad \text{if $|y|\le 1$}
        $$
        and
        $$
        \log(1 + |y - \beta|)\le \log(2+ \kappa) \qquad \text{if $|y|\le 1$}
        $$
        which implies that
        $$
        |v (\cdot + \beta)|_*^2 \le C |v|_*^2 + \log(2+\kappa)|v|_2^2 \le \tilde C \|v\|_X^2 \qquad \text{for all $v \in X$.}
        $$
        with $\tilde C = C + \log(2+\kappa)>1$. By translating $v$ and passing from $\beta$ to $-\beta$, this also implies that
        $$
        |v|_*^2  \le \tilde C \|v(\cdot+ \beta) \|_X^2 \qquad \text{for all $v \in X$.}
        $$
        Consequently, for every $u \in X \setminus \{0\}$ with $|\beta(u)| \le \kappa$ and every $v \in X$ we have 
$$
\|v\|_u = \|v (\cdot + \beta(u))\|_H^2 +  |v (\cdot + \beta(u))|_*^2 \le \|v\|_H^2 + \tilde C \|v\|_X^2 \le (1+\tilde C) \|v\|_X^2
$$
and, similarly,
$$
\|v\|_X^2 = \|v\|_H^2 + |v|_*^2 \le \|v (\cdot + \beta(u))\|_H^2 +  \tilde C \|v (\cdot + \beta(u))\|_X^2 \le (1+\tilde C) \|v(\cdot + \beta(u))\|_X^2 = (1+\tilde C)\|v\|_u
$$
Since $\beta$ is continuous and therefore locally bounded, these estimates imply \ref{M1}.
\end{proof}
\begin{corollary}\label{cor:unifequivalent}
	Let $u_n, u \in X$ be such that $u_n \to u$ strongly in $L^2(\R^2)$. Then, there is $C > 0$ such that
	$$ \frac{1}{C} \| v \|_{X} \leq \| v \|_{u_n} \leq C \| v \|_X \qquad \text{for all } v \in X,$$
	for all $n$ large enough.
\end{corollary}
\begin{proof}
	We just need to take a relative $L^2$-neighborhood $u \in U \subset X$, as in Lemma \ref{lemma:scalarproduct} and the inequalities hold for $n$ large enough that $u_n \in U$.
\end{proof}

\subsection{Action of subgroups of $E(2)$}
\label{sec:action-subgroups-e2}

In this section, we let $G$ be a closed subgroup of the group of euclidean motions $E(2)$ on $\R^2$.
The closedness assumption on $G$ implies the following: If $g_n:= g_{A_n,b_n} \in G$ are given with $g_n \to g_{A,b}$, i.e., $A_n \to A$ in $O(2)$ and $b_n \to b$ in $\R^2$, then $g_{A,b} \in G$.

For a given continuous homomorphism $\zeta : G \to \Z_2 := \{-1,1\}$ we then define the linear action of $G$ on either one of the function spaces $L^p(\R^2)$, $H$ and $X$ by 
\begin{equation}\label{eq:zeta_action}
	(g \ast_\zeta u)(x) := \zeta(g)u(g^{-1}x) \qquad g \in G.
\end{equation}
We will refer to $\eqref{eq:zeta_action}$ as the $\zeta$-action of $G$. This action is isometric on the spaces $L^p(\R^2)$ and $H$ but in general not on the space $X$, as $\|\cdot\|_X$ is not invariant under translations.  
We will mostly consider the case $\zeta \equiv 1$, and in this case we simply write $g \ast u$ in place of $g \ast_\zeta u$.

Observe that if $\zeta \not \equiv 1$ and $u \in X$ is invariant under the $\zeta$-action of $G$, then $u$ is a sign-changing function.

We first note the following direct consequence of (\ref{eq:homogeneity-scalar-prduct-family}) and Lemma~\ref{M3-satisfied-lemma-0}.

\begin{lemma}
\label{M3-satisfied-lemma}
The $\zeta$-action of $G$ defined in (\ref{eq:zeta_action}) satisfies condition $(M3)$ from Section~\ref{sec:LS theory}.  
\end{lemma}

Moreover, we have the following. 
\begin{lemma}
	\label{example-conditions-CC-G-*}
	In the setting above, the $\zeta$-action of $G$ satisfies the condition \ref{Gstar} from Section~\ref{sec:LS theory}.  
\end{lemma}

\begin{proof}
	Let $u,v \in X$, and we let $(g_n)_n$ be a sequence in $G$ with $g_n \ast_\zeta u \to v$ in $H$. Since $G$ acts isometrically on $H$, we then have $u=0$ if and only if $v=0$, and in this case we have $g \ast_\zeta  u =v$ for any $g \in G$.
	
	We may therefore assume $u \not = 0$, and we write $g_n:= g_{A_n,b_n}$ with $A_n \in O(2)$, $b_n \in \R^2$. From the convergence $g_n  \ast_\zeta  u \to v$ in $H$, it then follows that
	$$
	b_n + \beta(u) = \beta(g_n \ast u) =\beta(g_n  \ast_\zeta  u) \to \beta(v) \qquad \text{as $n \to \infty$,}
	$$
	where we recall in the second equality that $\beta(|u|) = \beta(u)$ for all $u \in X \setminus \{0\}$. Hence $b_n \to b:= \beta(v)-\beta(u)$. Moreover, after passing to a subsequence, we may assume that $A_n \to A$ in $O(2)$. Therefore $g_n \to g_{A,b}$ and $\zeta(g_n) \to \zeta(g_{A,b})$. Since the convergence $g_n \ast_\zeta  u \to v$ in $H$ also implies a.e.-convergence in $\R^2$, we have, for a.e. $x \in \R^2$,
	$$
	v(x)= \lim_{n \to \infty} (g_n  \ast_\zeta  u)(x)= \lim_{n \to \infty}\zeta(g_n) u (g_n^{-1} x)= \lim_{n \to \infty} \zeta(g_n) u (g_{A_n,b_n}^{-1} x)= \zeta(g_{A,b})u(g_{A,b}^{-1}x),
	$$
	which implies that $g_{A,b} \ast_\zeta  u = v$. We thus have proved \ref{Gstar}.
\end{proof}

Next we recall the following standard definitions.

\begin{definition}
	\label{group-elements-def}
	Let $g \in E(2)$.
	\begin{itemize}
		\item[(i)] We call $g$ a translation\footnote{Here $g$ is understood to be nontrivial, so $g\not =\id$} if $g=g_{b}$ for some $b \in \R^2 \setminus \{0\}$.
		\item[(ii)] We call $g$ a reflection if $g=g_{A,b}$ for some $A \in O(2)$, $b \in \R^2$ with $\det A = -1$ and $A b=-b$
		\item[(iii)] We call $g$ a rotation if $g=g_{A,b}$ for some $A \in O(2)$, $A \not =\id$ with $\det A = 1$ and $b \in \R^2$.
		\item[(iii)] We call $g$ a glide reflection if $g=g_{A,b}$ for some $A \in O(2)$, $b \in \R^2$ with $\det A = -1$ and $A b \not=-b$.  
	\end{itemize}
\end{definition}

Our aim is to apply the theory developed in Section~\ref{sec:LS theory} to $\Phi$ and $G$. Note however that we may not expect assumption \ref{Ic} to hold in general for $G$ and $\zeta$. For example if $\zeta \equiv 1$ and $G= \{\id, g\}$ with a reflection $g \in E(2)$, then there exists an infinite-dimensional subspace $X$ of functions $u$ satisfying $g \ast u=-u$. On the other hand, we have the following. 

\begin{lemma}
	\label{i-g-compatible}
If $u \in X \setminus \{0\}$ and $g \in E(2)$ is a translation or glide reflection, then $g* u \neq \pm u$.
\end{lemma}
\begin{proof}
Suppose by contradiction that $g* u = u$ or $g* u = - u$. Then we have $g^{2n}* u=u$ for every $n \in \Z$. If $g = g_{b}$ is a translation with some $b \in \R^2 \setminus \{0\}$, we have
	$g^{2n}=g_{\id,2n b}$ and hence
	$$
	u(x)= u(x+2n b) \qquad \text{for every $n \in \Z$ and a.e. $x \in \R^2$.}
	$$
	This contradicts the assumption that $u \in X \setminus \{0\}$. If $g = g_{A,b}$ is a glide reflection, we can write $b = b_1 + b_2$ with $A b_1 = -b_1$ and $Ab_2 = b_2$. Then
	we have $g^2 = g_{A^2,Ab+b}= g_{2b_2}$ so $g^2$ is a translation. As before, we then get 
	$$
	u(x)= u(x+4n b_2) \qquad \text{for every $n \in \Z$ and a.e. $x \in \R^2$.}
	$$
	Moreover, $b_2 \not =0$ since $Ab \not =-b$ by assumption. Again this contradicts the fact that $u \in X \setminus \{0\}$.
This concludes the proof.
\end{proof}

Finally, we recall the following standard version of Lions' lemma, see \cite{l}.
\begin{remark}\label{remark:H1boundedsequences}
	Let $(u_n)_n$ be bounded in $H^1(\mathbb{R}^2)$. Then, either $u_n \to 0$ in $L^s(\R^2)$ for all $s>2$, or there is a sequence of translations $g_n = g_{b_n}\in E(2)$, $b_n \in \R^2$ and $u \in H^1(\mathbb{R}^2) \setminus \{0\}$ such that, up to a subsequence, we have
	\begin{align*}
		g_n \ast u_n \rightharpoonup u &\quad \text{weakly in } H^1(\mathbb{R}^2), \\
		g_n \ast u_n \to u &\quad \text{in } L^s_{loc}(\mathbb{R}^2) \quad \text{for } s\geq 2,\\
		g_n \ast u_n \to u &\quad \text{a.e. in } \mathbb{R}^2.
	\end{align*}
\end{remark}

\section{High energy solutions of the indefinite logarithmic Choquard equation}\label{sec:choquard}
We continue using the notation from the previous section. Our aim is to apply the abstract theory in Section \ref{sec:LS theory} to obtain solutions to the Choquard equation
\begin{equation}\label{eq:choquard}
	- \Delta u + a(x) u + (\log|\cdot|\ast u^2)u = 0 \quad x \in \R^2,
\end{equation}
where $a \in L^\infty(\R^2)$ is invariant under the action of a closed subgroup $G \subset E(2)$.

For $\tau \ge 0$, we define the symmetric bilinear forms
\begin{align*}
	(u,v) & \mapsto B_1^\tau(u,v)= \int_{\R^2} \int_{\R^2}
	\log
	(e^{\tau}+|x-y|) u(x)v(y)\,dx dy,\\
	(u,v) & \mapsto B_2^\tau(u,v)= \int_{\R^2} \int_{\R^2}
	\log \Bigl(1+\frac{e^{\tau}}{|x-y|}\Bigr) u(x)v(y)\,dx dy,\\
	(u,v)& \mapsto B_0(u,v) \int_{\R^2} \int_{\R^2}
	\log(|x-y|)  u(x)v(y)\,dx dy, 
\end{align*}
noting that $B_0(u,v)=B_1^\tau(u,v)-B_2^\tau(u,v)$ for every $\tau \ge 0$.
Here the definition is restricted, in each case, to measurable
functions $u,v: \R^2 \to \R$ such that the corresponding double
integral is well defined in Lebesgue sense. Note that, since $0 \le \log
(1 +r) \le r$ for $r>0$, we have by the
Hardy-Littlewood-Sobolev inequality, see \cite{lieb}
\begin{equation}
	\label{eq:2}
	|B_2^\tau(u,v)| \le e^{\tau} \int_{\R^2} \int_{\R^2} \frac{1}{|x-y|} |u(x)v(y)|\,dx
	dy \le C_0 e^{\tau} |u|_{\frac{4}{3}} |v|_{\frac{4}{3}} \qquad \text{for $u,v
		\in L^{\frac{4}{3}}(\R^2)$}
\end{equation}
with a constant $C_0>0$.  We
now define, for $\tau \ge 0$, the functionals
\begin{align*}
	V_1^\tau&: H^1(\R^2) \to [0,\infty],&&\qquad V_1^\tau(u)=B_1^\tau(u^2,u^2)= \int_{\R^2} \int_{\R^2}
	\log
	(e^{\tau}+|x-y|) u^2(x)u^2(y)\,dx dy,\\
	V_2^\tau&: L^{\frac{8}{3}}(\R^2) \to [0,\infty) ,&&\qquad V_2^\tau(u)=B_2^\tau(u^2,u^2)= \int_{\R^2} \int_{\R^2}
	\log
	\Bigl(1+\frac{e^{\tau}}{|x-y|}\Bigr) u^2(x)u^2(y)\,dx dy,\\
	V_0&: H^1(\R^2) \to \R  \cup \{\infty\},&&\qquad V_0(u)=B_0(u^2,u^2)= \int_{\R^2} \int_{\R^2}
	\log(|x-y|)  u^2(x)u^2(y)\,dx dy.
\end{align*}
Note that, as a consequence of \eqref{eq:2}, we
have
\begin{equation}
	\label{eq:3}
	|V_2^\tau(u)| \le C_0 e^{\tau}  |u|_{\frac{8}{3}}^4 \qquad \text{for
		all $u \in L^{\frac{8}{3}}(\R^2)$,}
\end{equation}
so $V_2^\tau$ only takes finite values on $L^{\frac{8}{3}}(\R^2)$. Moreover, we have
\begin{equation}\label{eq:V1tau}
	V_1^\tau(u) \ge \tau |u|_{2}^4 \qquad \text{for every $\tau \in \R$, $u \in X$.}
\end{equation}
Since for $\tau \ge 0$ we have 
\begin{align*}
	\log\left(e^{\tau}+|x-y|\right) &\leq \log\left(e^\tau +|x|+|y|\right) \le \log\left(e^{\tau}(1 +|x|)(1 +|y|)\right) \\
	&\leq \tau + \log\left(1+|x|\right) + \log\left(1+|y|\right)
	\quad \text{for}\ x, y \in \mathbb{R}^{2},
\end{align*}
we get the estimate
\begin{align}\label{2-3}
	B_{1}^\tau(uv, wz) &\leq \int_{\mathbb{R}^{2}}\int_{\mathbb{R}^{2}}
	\left[\tau \log\left(1+|x|\right) + \log\left(1+|y|\right) \right]
	\left|u(x)v(x)\right|\left|w(y)z(y)\right|dxdy \notag\\
	&\leq \tau |u|_2|v|_2|w|_2|z|_2+  |u|_{*}|v|_{*}|w|_{2}|z|_{2}+|u|_{2}|v|_{2}|w|_{*}|z|_{*}
	\quad \text{for}\ u,v,w,z \in L^{2}(\mathbb{R}^{2})
\end{align}
with the conventions $\infty \cdot 0 = 0$ and $\infty \cdot s =\infty$
for $s>0$.

The following three lemmas are given in \cite[Lemmas 2.1, 2.2 and 2.6]{CW} for the special case $\tau = 0$. The proofs are almost the same for general $\tau \ge 0$, so we skip them.  

\begin{lemma}
	\label{sec:comp-cond-2}
	Let $(u_n)_n$ be a sequence in $L^2(\R^2)$ such that $u_n \to u \in
	L^2(\R^2) \setminus \{0\}$ pointwisely a.e. on $\R^2$. Moreover, let
	$(v_n)_n$ be a bounded sequence in $L^2(\R^2)$ such that
	\begin{equation*}
		\label{eq:19}
		\sup \limits_{n \in \N} B_1^\tau(u_n^2,v_n^2)<\infty \qquad \text{for some $\tau \in \R$.}
	\end{equation*}
	Then there exists $n_0 \in \N$ and $C>0$ such that $|v_n|_*<C$ for $n \ge n_0$.\\
	If, moreover,
	\begin{equation*}
		\label{eq:18}
		B_1^\tau(u_n^2,v_n^2) \to 0 \quad \text{and}\quad |v_n|_2 \to 0\qquad \text{as $n \to \infty$,}
	\end{equation*}
	then
	\begin{equation*}
		\label{eq:14}
		|v_n|_* \to 0 \qquad \text{as $n \to \infty$, $n \ge n_0$.}
	\end{equation*}
\end{lemma}

\begin{lemma}
	\label{sec:comp-cond}
	\begin{itemize}
		\item[(i)] The space $X$ is compactly embedded in $L^s(\R^2)$ for all
		$s \in [2,\infty)$.
              \item[(ii)] For every $\tau \ge 0$, the functionals $V_0,V_1^\tau,V_2\tau$ are of class $C^2$ on $X$.\\
                Moreover, $(V_i^\tau)'(u)v = 4 B_i(u^2,uv)$ for $u,v \in X$ and $i=0,1,2$.
		\item[(iii)] $V_2^\tau$ is continuous on $L^{\frac{8}{3}}(\R^2)$ for every $\tau \ge 0$.
		\item[(iv)] $V_1^\tau$ is weakly lower semicontinuous on $H^1(\R^2)$ for every $\tau \ge 0$.
	\end{itemize}
\end{lemma}	

\begin{lemma}\label{lemma:B1weakconvergence}
	Let $(u_n)_n$, $(v_n)_n$, $(w_n)_n$ be bounded sequences in $X$ such that $u_n \rightharpoonup u$ weakly in $X$. Then, for every $z \in X$, we have $B_1^\tau(v_n w_n, z(u_n - u)) \rightarrow 0$ as $n \rightarrow \infty$.
\end{lemma}

Now, for fixed $a \in L^{\infty}(\R^2)$ we define the bilinear form
$$
q_a(u,v)= \int_{\R^2} \nabla u \cdot \nabla v + \int_{\R^2}auv  \qquad u,v \in H^1(\R^2),
$$
and for short we write $q_a(u) := q_a(u,u)$ to denote the associated quadratic form. We also define 
$$ [u]_a^2 := \int_{\R^2} a(x)u^2 dx \qquad u \in H^1(\R^2). $$

Then we have the energy functional
\begin{equation}
  \label{eq:def-Phi-Choquard}
\Phi \in C^1(X,\R),\qquad  \Phi(u) = \frac{1}{2}q_a(u) + \frac{1}{4} V_0(u)
\end{equation}
whose critical points are weak solutions to \eqref{eq:choquard}.

For all $u \in X$, we define the fiber maps
\begin{equation}\label{eq:fibermaps}
	f_u: [0,\infty) \to \R, \quad f_u(t)= \Phi(t u)= \frac{t^2}{2}q_a(u) + \frac{t^4}{4} V_0(u).
\end{equation}

Finally, we introduce the Nehari manifold
$$
\cN:=  \{u \in X \setminus \{0\} \::\: \Phi'(u)u = 0\} = \{u \in X \setminus \{0\} \::\: f_u'(1)=0\} =  \{u \in X \setminus \{0\} \::\: q_a(u) = -V_0(u)\} .
$$
Then we split it in the subsets
$$
\cN_-:= \{u \in \cN\::\:  V_0(u) < 0 \}, \quad \cN_+:= \{u \in \cN\::\:  V_0(u) > 0 \}
$$
and
$$
\cN_0:= \{u \in \cN\::\: V_0(u) = 0\} = \{u \in X \setminus \{0\} \::\: q_a(u) = V_0(u)= 0\}.
$$

\begin{remark}\label{remark:propertiesg}
	Observe that there is a unique $t_u > 0$ such that $f_u'(t_u) = 0$ if and only if $q_a(u)V_0(u) < 0$. Indeed, this is clear from the identities
	\begin{equation*}\label{eq:derivativefiber}
		f_u'(t)=  q_a(u)+t^2V_0(u) \quad \text{and} \quad f_u''(t) = 2V_0(u)t,
	\end{equation*}
	in which case, $t_u$ is a global maximum for $V_0(u) < 0$ and a global minimum for $V_0(u) > 0$. The number $t_u$ is exactly
	$$ t_u^2 = - \frac{q_a(u)}{V_0(u)}. $$
	Then, we can define a map $\sigma(u):= t_u u$ on the set $\mathcal{O} := \{u \in X : q_a(u)V_0(u) < 0\}$, which is a $C^2$-function on $\mathcal{O}$ and satisfies $\sigma(u) \in \mathcal{N}$.
\end{remark}

\begin{remark}
 $\mathcal{N}_{\pm}$ is a $C^2$-manifold of codimension $1$ and locally closed in $X$. Indeed, observe that the map $J(u):= q_a(u) + V_0(u)$ satisfies $\mathcal{N} = J^{-1}(0)$ and 
 $$ J'(u)u = 2q_a(u) + 4V_0(u)u = 2 V_0(u) \neq 0 \quad \text{for } u \in \mathcal{N}_{\pm}$$
 so $0$ is a regular value of $J$. Moreover, $\mathcal{N}_{\pm}$ are natural constraints for $\Phi$, i.e. critical points of the restriction $\Phi|_{\mathcal{N}_{\pm}}$ are critical points of the entire functional $\Phi:X \to \R$, see \cite[Proposition 2.5]{c}.
\end{remark}

To state the main result of this section, we now introduce the main equivariance assumption. So, as in Section~\ref{sec:action-subgroups-e2} we consider a closed subgroup $G \subset E(2)$, and we assume that the potential $a \in L^\infty(\R^2)$ in \ref{sec:choquard} is $G$-invariant, i.e. we have $a \circ g = a$ for every $g \in G$.
Moreover, we consider the $\zeta$-action of $G$ on $X$ as defined in (\ref{eq:zeta_action}) for some given continuous homomorphism $\zeta : G \to \Z_2 := \{-1,1\}$.  It is easy to see from that the functional $\Phi:X \to \R$ is then invariant under the $\zeta$-action of $G$. Finally, we let $G_0$ be a fixed compact normal subgroup of $G$, and we say that a function $u \in X$ is $(G_0,\zeta)$-invariant if $g \ast_\zeta u= u$ for every $g \in G_0$. We also recall the subspace of $(G_0,\zeta)$-invariant functions
$$ X_0 := \{u \in X \::\: g \ast_\zeta u = u \quad \text{for all } g \in G_0 \}.$$

\begin{definition}
  \label{admissible-triple}
The pair $(G,G_0)$ will be called admissible if one of the following is satisfied.  
\begin{enumerate}[label*=$(G_\arabic*)$]
	\item \label{G1} $G = G_0 \subset O(2)$, and $G \not \subset \{id, g\}$ for any reflection $g \in E(2)$.
	\item \label{G3} $G = \langle g_{b_1} , g_{b_2} \rangle$ where $b_1,b_2 \in \R^2$ are linearly independent vectors, and $G_0 = \{id\}$.
	\item \label{G4} $G = \langle g_{A}, g_{b} \rangle$, where $g_{A,b}$ is a glide reflection, and $G_0 = \langle g_{A,0} \rangle$.
\end{enumerate}
\end{definition}
We point out that \ref{G1} includes the case where $G = G_0 = \langle g_A \rangle$ for a rotation $A \in SO(2)$ of angle $\pi/m$, $m \in \N$. In this case we may define $\zeta(g_A^k) = (-1)^{k}$ for $k = 1,\ldots,2m$, which implies that every $(G_0,\zeta)$-invariant function $u \in X \setminus \{0\}$ is sign changing and nonradial. 

The main result of this section is the following.

\begin{theorem}\label{thm:multiplicityN-}
	Let $a \in L^\infty(\R^2)$ be $G$-invariant and assume that the pair $(G,G_0)$ is admissible. Then the equation \eqref{eq:choquard} has a sequence $(\pm u_n)_n$ in $\mathcal{N}_{-}$ of nontrivial $(G_0,\zeta)$-invariant solution pairs with $\Phi(\pm u_n) \to \infty$ as $n \to \infty$. Moreover, if $\einf_{\R^2} a > 0$, then
	$$ \Phi(\pm u_1) = \inf \{ \Phi(u) : u \in X_0 \setminus \{0\}, \, \Phi'(u) = 0\} $$
\end{theorem}

Before we start with the proof of this theorem, we first deduce the main results of the introduction from Theorem~\ref{thm:multiplicityN-}.

\begin{proof}[Proof of Theorems \ref{thm:multiplicity_translation} - \ref{thm:multiplicity_reflection}]
  Theorem~\ref{thm_multiplicity_radial} follows from Theorem~\ref{thm:multiplicityN-} by considering $G = G_0 = O(2)$ and the constant homomorphism $\zeta \equiv 1$. In this case, the admissibility condition \ref{G1} is satisfied, and a $(G_0,\zeta)$-invariant function is merely a radially symmetric function.
  
To deduce Theorem~\ref{thm_multiplicity_nonradial} from Theorem~\ref{thm:multiplicityN-}, we let $g \in O(2)$ be a rotation of angle $\pi/m$. Moreover, we set $G = G_0 := \{ g^{j} \}_{j=1}^{2m}$, and we define an homomorphism $\zeta:G \to \Z_2$ by setting $\zeta(g^j) := (-1)^{j}$. Again, the admissibility condition \ref{G1} is satisfied in this case, and every $(G_0,\zeta)$-invariant function is nonradial and sign-changing. 

Next, for the proof of Theorem~\ref{thm:multiplicity_translation},  it suffices to consider the group $G$ of translations by vectors in $\Z^2$, and $G_0 = \{id\}$, $\zeta \equiv 1$. Then the admissibility condition \ref{G3} is satisfied.  

For the last case, Theorem~\ref{thm:multiplicity_reflection}, we let $G \subset E(2)$ be the group generated by the reflection $h$ at the $x_1$-axis and by the translation $x \mapsto x+ e_1$, where $e_1$ is the first coordinate vector. Moreover, we let $G_0$ be the normal subgroup generated by $h$, and we consider $\zeta \equiv 1$. Then the admissibility condition \ref{G4} is satisfied, and a $(G_0,\zeta)$-invariant function $u \in X$ satisfies $u(x_1,-x_2) = u(x_1,x_2)$ for $x \in \R^2$.
\end{proof}
The remainder of this section is devoted to the proof of Theorem~\ref{thm:multiplicityN-}. So from now on, we fix an admissible pair $(G,G_0)$ of subgroups of $E(2)$ in the sense of Definition~\ref{admissible-triple}. We wish to apply Theorem~\ref{c-k-to-infty-variant} with $M = \cN_-$ and the restriction $\Phi\Big|_{\cN_-}$ functional $\Phi$ defined in (\ref{eq:def-Phi-Choquard}).

In the following subsections, we verify the assumptions of Theorem~\ref{c-k-to-infty-variant}. 

\subsection{Completeness of energy intervals}
\label{sec:local-completeness}

The aim of this short subsection is to prove the following. 
\begin{lemma}\label{lemma:dcomplete}
	$\mathcal{N}_- \cap \Phi^{-1}([a,b])$ is complete for all $0<a \leq b < \infty$ with respect to the metric $d$ defined in \eqref{eq:definition-metric-d}.
      \end{lemma}
      The following basic inequality will be needed in the proof of this lemma and also at a later stage:
      
\begin{equation}
\label{inequalitymetric}
d(u,v) \geq \|u-v\|_H \geq |u-v|_2 \qquad \text{for all } u,v \in \mathcal{N}_-.
\end{equation}
To see this, if suffices to consider a $C^1$-path $\gamma :[0,1] \to \mathcal{N}_-$ from $u$ to $v$ and to note that, by (\ref{eq:uniform-lower-bounds})
	\begin{align*}
		\int_0^1 \| \gamma'(s) \|_{\gamma(s)} ds \ge \int_0^1 \| \gamma'(s) \|_{H} ds \geq \|u-v\|_H \ge |u-v|_2.
	\end{align*}
        Taking the infimum over all such paths in this inequality, we get (\ref{inequalitymetric}).

\begin{proof}[Proof of Lemma~\ref{lemma:dcomplete}]
  Let $(u_n)_n$ be a Cauchy sequence in $\mathcal{N}_- \cap \Phi^{-1}([a,b])$. By (\ref{inequalitymetric}), $(u_n)_n$ is a Cauchy sequence in $H$ and therefore $u_n \to u$ strongly in $H$ and therefore also in $L^2(\R^2)$. We claim that
$u \not = 0$. Suppose by contradiction that $u = 0$; then also $q_a(u_n) \to 0$ and $V_2^0(u_n) \to 0$ by (\ref{eq:3}) and Sobolev embeddings. Since $0 \le V_1^0(u_n) < V_2^0(u_n)$ by definition of $\mathcal{N}_-$, it then also follows that $V_1^0(u_n) \to 0$ and therefore $V_0(u_n) \to 0$, which implies that $\Phi(u_n) \to 0$ as $n \to \infty$. This contradicts the assumption $a>0$ in the statement of the lemma. Hence $u \not = 0$, as claimed. Now Lemma \ref{lemma:scalarproduct} guarantees the uniform equivalence of the norms $\|\cdot\|_{v}$ for $v$ in an $L^2$-neighborhood of $u$. Then, we may argue as in Lemma \ref{compatability-lipschitz}\footnote{See the argument preceding \eqref{eq:lower_bound_d}.} to find a constant $C>0$ such that
	\begin{align*}
		\int_0^1 \| \gamma_{n,m}'(s) \|_{\gamma_{n,m}(s)} ds \geq  C \| u_n - u_m \|_{X}.
	\end{align*}
	Thus $(u_n)_n$ is a Cauchy sequence in $X$ as well, which implies that $u_n \to u$ strongly in $X$. Consequently, $u \in \Phi^{-1}([a,b])$ by the continuity of $\Phi$ in $X$. Moreover, $u \in \mathcal{N}_- \cup \mathcal{N}_0$ since $\mathcal{N}_- \cup \mathcal{N}_0$ is relatively closed in $X \setminus \{0\}$. However, $\Phi \equiv 0$ on $\mathcal{N}_0$, and therefore $u \in \mathcal{N}_- \cap \Phi^{-1}([a,b])$ since $a>0$.
\end{proof}

\subsection{Verification of the $(NCG)$-condition}
\label{sec:verif-ncg-cond}
This subsection is devoted to the verification of condition $(NCG)_c$, as introduced in Subsection~\ref{sec:abstr-sett-main}, in the present context where $M = \cN_-$ and $\Phi$ is defined in (\ref{eq:def-Phi-Choquard}). We start with the following useful observation.

\begin{lemma}\label{lemma:V0'notzero}
	For all $u \in X \setminus \{0\}$ we have $V_0'(u) \neq 0$ in $X'$.
\end{lemma}
\begin{proof}
	Arguing by contradiction, let us assume that $V_0'(u) \equiv 0$, i.e., $V_0'(u) = 4 B_0(u^2,uv)=0$ for every $v \in X$. Since this holds, in particular, for $v \in C^\infty_c(\R^2)$, the fundamental lemma of the calculus of variations implies that $(\log|\cdot|\ast |u|^2)u \equiv 0$. And, since $(\log|\cdot|\ast|u|^2)(x) \to \infty$ as $|x| \to \infty$, it follows that $u$ is compactly supported; from this we will obtain a contradiction. Indeed, setting $u_\varepsilon := u(\varepsilon^{-1} \cdot)$, we have, on the one hand
	\begin{align}\label{eq:dV0pos}
		\left. \frac{d}{d\varepsilon}\right|_{\varepsilon = 1} V_0(u_\varepsilon) =\left. \frac{d}{d\varepsilon}\right|_{\varepsilon = 1} \varepsilon^4(\log \varepsilon |u|_2^4 + V_0(u)) = |u|_2^4 > 0.
	\end{align}
	On the other hand, from the dominated convergence theorem we derive the identity
	\begin{align*}
		\left. \frac{d}{d\varepsilon}\right|_{\varepsilon = 1} V_0(u_\varepsilon) = -4 \int_{\mathbb{R}^{2}} \int_{\mathbb{R}^{2}} \log|x-y| u(x)^2u(y)(\nabla u(y) \cdot y) \, dx dy.
	\end{align*}
	Now, define $g(y) := \nabla u(y) \cdot y$ and let $(\varphi_n)_n \subset C^{\infty}_c(\mathbb{R}^2)$ such that $\varphi_n \to g$ in $L^2(\mathbb{R}^2)$. Observe that, since $u$ has compact support, we also have $\varphi_n \to g$ with respect to $|\cdot|_*$. Then, from \eqref{eq:2} we have
	\begin{align}\label{eq:B2estimate}
		|B_2(u^2,u(g-\varphi_n))| \leq C |u|_{8/3}^2|u(g - \varphi_n))|_{4/3} \leq C |u|_{8/3}^2 |u|_4 |g-\varphi_n|_2 = o(1),
	\end{align}
	and then, from \eqref{2-3} we have
	\begin{align}\label{eq:B1estimate}
		|B_1(u^2,u(g-\varphi_n))| \leq |u|_*^2|u|_2|g-\varphi_n|_2 + |u|_2^2|u|_*|g - \varphi_n|_* = o(1).
	\end{align}
	Thereby, we conclude from \eqref{eq:B1estimate} and \eqref{eq:B2estimate} that 
	\begin{align*}
		0 = V_0'(u)\varphi_n = B_0(u^2,ug) = -4 \int_{\mathbb{R}^{2}} \int_{\mathbb{R}^{2}} \log|x-y| u(x)^2u(y)(\nabla u(y) \cdot y) \, dx dy + o(1)
	\end{align*}
	which contradicts \eqref{eq:dV0pos}.
\end{proof}

Our first step in our verification of the \ref{NCG0}-condition is the following criterion for compactness modulo the $G$-action.

\begin{proposition}
	\label{PS-application}
	Let $(G,G_0)$ be an admissible pair, and let $(u_n)_n \subset X_0$ be a sequence which satisfies, for some $\tau > 0$, the following conditions:
	\begin{itemize}
		\item[(i)] $(u_n)_n$ is bounded in $H$.
		\item[(ii)] $V_1^\tau(u_n)$ remains bounded as $n \to \infty$ for all $\tau \ge 0$.   
		\item[(iii)] $\liminf \limits_{n \to \infty}V_2^\tau(u_n)>0$ for some $\tau \ge 0$.
	\end{itemize}
	Then there exists $u \in X \setminus \{0\}$ and a sequence $(g_n)_n$ in $G$ with the property that, after passing to a subsequence, $g_n \ast_\zeta  u_n \rightharpoonup u$ weakly in $X$.
\end{proposition}

\begin{proof}
It suffices to prove that $g_n* u_n \rightharpoonup u$ weakly in $X$ for suitable $u \in X \setminus \{0\}$, $g_n \in G$, since then $g_n \ast_\zeta  u_n \rightharpoonup u$ or $g_n \ast_\zeta  u_n \rightharpoonup -u$ in $X$ after passing to a subsequence. 

By assumption (iii), we may pass to a subsequence with
\begin{equation}
	\label{eq:iii-variant}
	\inf_{n \to \infty}V_2^\tau(u_n)>0.    
\end{equation}
Moreover, by assumption (i) and Remark \ref{remark:H1boundedsequences}, either $u_n \to 0$ in $L^s(\R^2)$ for all $s>2$, or there exists a sequence of translations $g_n := g_{b_n}$ and $u \in H^1(\R^2) \setminus \{0\}$ with the property that, after passing to a subsequence,
we have 
\begin{align*}
	g_{b_n}* u_n \rightharpoonup u &\quad \text{weakly in } H^1(\mathbb{R}^2), \\
	g_{b_n} * u_n \to u &\quad \text{in } L^s_{loc}(\mathbb{R}^2) \quad \text{for } s\geq 2,\\
	g_{b_n}* u_n \to u &\quad \text{a.e. in } \mathbb{R}^2.
\end{align*}
The first case can be ruled out by (\ref{eq:iii-variant}) and (\ref{eq:3}). Then, from $(ii)$ it follows using Lemma \ref{sec:comp-cond-2} that $(g_{b_n} \ast u_n)_n$ is bounded in $X$. Then, from Fatou's lemma we conclude that $|u|_* \le \liminf \limits_{n \to \infty}|g_{b_n} \ast u_n|_* < \infty$ and therefore $u \in X$.\\

We start with \ref{G1}. First we show that the sequence $(b_n)_n$ is bounded. In case \ref{G1}, $G$ must contain at least one rotation $g= g_{A}$ with $A \in O(2)$, $A \not =\id$ with $\det A = 1$. Next, let $R>0$ be such that 
$$\int_{B_R(0)}u^2\,dx >0,$$
and suppose by contradiction that
$|b_n| \to \infty$ after passing to a subsequence. Then we have $r_n:= \dist(B_R(-b_n),B_R(-A^{-1}b_n))\to \infty$ as $n \to \infty$, and the $G_0$-invariance of $u_n^2$ implies that
\begin{align*}
	V_1(u_n)&= \int_{\R^2}\int_{\R^2} \log (1+|x-y|)u_n^2(x)u_n^2(y)\,dxdy\\
	&\ge \int_{B_R(-b_n)}\int_{B_R(-A^{-1}b_n)} \log (1+|x-y|)u_n^2(x)u_n^2(y)\,dxdy\\  
	&\ge \log (1+r_n) \Bigl(\int_{B_R(-b_n)} u_n^2(x)dx\Bigr)^2,
\end{align*}
where, by Fatou's Lemma
\begin{align*}
	\liminf_{n \to \infty} \int_{B_R(-b_n)} u_n^2(x)dx&=\liminf_{n \to \infty} \int_{B_R(0)} u_n^2(x-b_n)dx\\
	&= \liminf_{n \to \infty} \int_{B_R(0)} (g_{b_n} * u_n)^2 dx \ge \int_{B_R(0)}u^2\,dx >0
\end{align*}
This contradicts assumption (ii). Thus $(b_n)_n$ is bounded and, after passing to a subsequence, we may assume that $b_n \to b$ in $\R^2$. Thereby, we see that $u_n \rightharpoonup g_{b}^{-1} \ast u$ weakly in $X$, i.e. $g_{n} = id$ for all $n \in \mathbb{N}$.

Now, for case \ref{G3}, we assume without loss of generality that $b_1 = (1,0), b_2 = (0,1)$. We define
$$\lfloor x \rfloor := \max\{z \in \mathbb{Z} : z \leq x \} \quad \text{ for } x\in \R \quad \text{and} \quad \lfloor z \rfloor := (\lfloor z_1 \rfloor, \lfloor z_2 \rfloor) \quad \text{for } z = (z_1,z_2) \in \R^2,$$
so that $g_{\lfloor z \rfloor} \in G$ for all $z \in \R^2$. Now, we have
$ | z - \lfloor z \rfloor | \leq \sqrt{2} $ for all $z \in \R^2$ and thus setting $h_n := g_{\lfloor b_n \rfloor}$ we see that
\begin{align*}
&\left| | h_n \ast u_n|_*^2  - |g_n \ast u_n |_*^2 \right| \leq \int_{\R^2} \left| \log(1+|x+\lfloor b_n \rfloor|) - \log(1+|x+b_n|) \right| u_n^2(x) dx \\
&\leq \int_{\R^2} \left| |x + \lfloor b_n \rfloor | - |x + b_n| \right| u_n^2(x) dx \leq \int_{\R^2} |b_n - \lfloor b_n \rfloor| u_n^2(x) dx \leq \sqrt{2} |u_n|_2^2
\end{align*}
Then $(h_n \ast u_n)_n$ is bounded in $X$ as well, since $(u_n)_n$ is bounded in $L^2(\R^2)$. Thus, passing to a subsequence we have $h_n \ast u_n \rightharpoonup g_{b}^{-1} \ast u \neq 0$ weakly in $X$, where $b \in \R^2$ is an accumulation point of the bounded sequence $(b_n - \lfloor b_n \rfloor)_n$.\\

Finally, for case \ref{G4}, we assume without loss of generality, that $G_0 = \{id, g\}$ with $g=g_{A,0}$, where $A \in O(2)$ is the reflection of the $x_2$-coordinate. By assumption, $G$ contains a translation $g_{\id,c}$ with $A c \not= -c$, i.e., $c_1 \not =0$.
Without loss of generality, we may assume that $c= c_1e_1$, since $g_{A,0} g_{\id c} g_{A,0} g_{\id,c}= g_{\id, 2c_1 e_1}$.
Similarly as in the proof with case \ref{G1} above, we may now exclude that for $b_n= (b_n^1,b_n^2)$ we have
$|b_n^2| \to \infty$, so, after passing to a subsequence, $|b_n^2|$ remains bounded. Hence we may obtain the assertion similarly as with \ref{G3} by translating along $c$ in the $x_1$-direction with $\tilde{g}_{\tilde{b}_n} \in G$, where the $\tilde{b}_n$ lie on the $x_1$-axis.
\end{proof}

\begin{remark}\label{remark:ceramiseq}
	We recall the function defined in \eqref{eq:def-N-w}
	$$ N(u) := \inf\{ \|w \|_w+ d(u,w) : w \in \mathcal{N} \}.$$
	Observe that, from Lemma~\ref{inequalitymetric}
	$$ \|w\|_w  + d(u,w) \geq |w|_2 + |u-w|_2 \geq |u|_2 \quad \text{for all } w \in \mathcal{N}_-,$$
	and thus, for all $u \in \mathcal{N}_-$
	$$ N(u) \geq |u|_2. $$
	As a consequence, if $(u_n)_n \subset \mathcal{N}_-$ is a $(NC)_c$-sequence for $\Phi$, then
	$$ \| \nabla_{u_n} \Phi(u) \|_{u_n}(1 + |u_n|_2) \to 0.$$ 
\end{remark}

\begin{proposition}\label{prop:PSC23}
	Let $(G,G_0)$ be an admissible pair, let $c \in \R \setminus \{0\}$, and let $(u_n)_n \subset \mathcal{N} \cap X_0$ be a sequence with
  $$
  \Phi(u_n) \to c \quad \text{and}\quad \|\nabla_{u_n} \Phi(u_n)\|_{u_n} (1+ |u_n|_2) \to 0 \qquad \text{as $n \to \infty$.}
  $$
Then, there are $g_n \in G$ and $u \in \mathcal{N} \cap X_0$ such that, passing to a subsequence, $g_n \ast_\zeta u_n \to u$ strongly in $X$. 
\end{proposition}

If, under the assumptions of Proposition~\ref{prop:PSC23}, we have $u_n \in X_0$ for all $n \in \N$, then also the functions
$g_n \ast_{\zeta} u_n \in X_0$  for all $n \in \N$ by (\ref{eq:G-0-invariance-X_0}) applied to the action $\ast = \ast_{\zeta}$.
Hence also $u = \lim \limits_{n \to \infty}g_n \ast_{\zeta} u_n \in X_0$ in this case.

\begin{proof}
  Let $J : X \to \R$ be given by $J(u) = q_a(u) + V_0(u)$, so that $\mathcal{N} = J^{-1}(0)$ and therefore
  $0= J(u_n)= q_a(u_n) + V_0(u_n)$ for all $n \in \N$, which implies that
  \begin{equation}
    \label{eq:constraint-condition-1}
  	0 \not = c + o(1) = \Phi(u_n) = \frac{1}{4}q_a(u_n)= -\frac{1}{4}V_0(u_n).
  \end{equation}
We also note that, for every $n \in \N$, there exists $\mu_n \in \R$ such that
	\begin{align}\label{eq:J}
	\begin{split}
		\Phi'(u_n)v &= \Phi'(u_n)[P_{u_n}v] + \mu_n J'(u_n)[v] \\
		&= \langle \nabla_{u_n} \Phi(u_n),v \rangle_{u_n} + \mu_n J'(u_n)[v]
		\quad \text{for all } v \in X,
	\end{split}
	\end{align}
        We first claim that $(u_n)$ is bounded in $L^2(\R^2)$. Suppose by contradiction that $|u_n|_2 \to \infty$ after passing to a subsequence. We then consider the sequence of normalized functions $\hat{u}_n := u_n /|u_n|_2$, $n \in \N$. By (\ref{eq:constraint-condition-1}), we then have
\begin{align*}
	c + o(1) = \Phi(u_n) = \frac{1}{4}q_a(u_n) = \frac{1}{4} |u_n|_2^2\Bigl(|\nabla \hat{u}_n|_2^2 + \int_{\R^2}a(x){\hat{u}_n}^2\,dx\Bigr)) \ge \frac{1}{4} |u_n|_2^2\Bigl(|\nabla \hat{u}_n|_2^2 - |a|_{\infty}\Bigr)),
\end{align*}
which, since $|u_n|_2 \to \infty$, implies that $\limsup \limits_{n \to \infty}|\nabla \hat{u}_n|_2^2 \le |a|_{\infty}$ and therefore $(\hat{u}_n)_n$ is bounded in $H^1(\R^2)$. Consequently, also $\Bigl(V_2^\tau(\hat{u}_n)\Bigr)_n$ remains bounded for every $\tau \ge 0$ as a consequence of (\ref{eq:3}) and Sobolev embeddings. Moreover, by (\ref{eq:constraint-condition-1}) we have 
\begin{equation}
  \label{eq:limit-V-0-zero-limit}
V_0(\hat{u}_n) = |u_n|_2^{-4}V_0(u_n)= - |u_n|_2^{-4}q_a(u_n) = -|u_n|_2^{-4}4 \Phi(u_n)= |u_n|_2^{-4}(-4 c + o(1)) \to 0
\end{equation}
as $n \to \infty$, which then also implies that $(V_1^\tau(\hat{u}_n))_n$ remains bounded for every $\tau \ge 0$. If $V_2^\tau(\hat{u}_n) \to 0$ as $n \to \infty$ for all $\tau \ge 0$, then 
$$
\liminf_{n \to \infty} V_0(\hat{u}_n) = \liminf_{n \to \infty} V_1^\tau(\hat{u}_n) \geq \tau \qquad \text{for all $\tau > 0$}
$$
according to \eqref{eq:V1tau}, which contradicts (\ref{eq:limit-V-0-zero-limit}). Hence, after passing to a subsequence, we have $\liminf \limits_{n \to \infty} V_2^\tau(\hat{u}_n)>0$ for some $\tau \ge 0$. Hence the sequence $(\hat u_n)_n \subset X$ satisfies the assumptions of Proposition~\ref{PS-application}, and therefore there exists $\hat{u} \in X \setminus \{0\}$ and $g_n \in G$, $n \in \N$ with
$g_n \ast_{\zeta} \hat{u}_n \to \hat{u}$ weakly in $X$ and therefore strongly in $L^2(\R^N)$. Without loss of generality, we may assume that $g_n = \id$ for all $n \in \N$ in the following, otherwise we replace $\hat{u}_n$ by $g_n \ast_\zeta {\hat{u}_n}$.  Here we note that, since $\hat{u}_n \in X_0$ for all $n \in \N$, we also have $g_n \ast_{\zeta} u_n \in X_0$  for all $n \in \N$ by (\ref{eq:G-0-invariance-X_0}).

As a consequence of Corollary \ref{cor:unifequivalent} we see that $\|\cdot\|_{u_n}= \|\cdot \|_{\hat{u}_n}$ is uniformly equivalent to $\|\cdot\|_X$ for every $n \in \N$, and therefore
\begin{equation}
  \label{eq:uniform-hat-un-bound}
\|\hat{u}_n\|_{u_n} \le C \qquad \text{for all $n \in \N$ with a constant $C>0$.}
\end{equation}
Moreover, by the Cerami property and (\ref{eq:uniform-hat-un-bound}), we have
        \begin{align*}
          | \langle \nabla_{u_n} \Phi(u_n),u_n \rangle_{u_n}| &\le \|\nabla_{u_n} \Phi(u_n)\|_{u_n} \|u_n\|_{u_n}\le
                                                                \|\hat u_n\|_{u_n} \|\nabla_{u_n} \Phi(u_n)\|_{u_n} |u_n|_2 \\
          &\le C \|\nabla_{u_n} \Phi(u_n)\|_{u_n} |u_n|_2 \to 0 \qquad \text{as $n \to \infty$,}
            \end{align*}
which by (\ref{eq:J}) and (\ref{eq:constraint-condition-1}) implies that
\begin{align*}
0=		q_a(u_n)+ V_0(u_n)  &= \Phi'(u_n)u_n = o(1) + \mu_n  J'(u_n)u_n \\
		&=  o(1) + \mu_n ( 2q_a(u_n) + 4 V_0(u_n))= o(1) + \mu_n (-8c + o(1)).
	\end{align*}
        Hence $\mu_n \to 0$ as $n \to \infty$. For all $v \in X$, we thus have, by (\ref{eq:J}) and since $\|\nabla_{u_n} \Phi(u_n)\|_{u_n} \to 0$ as $n \to \infty$,
\begin{align*}
q_a(u_n,v)+ B_0(u_n^2,u_n v)  &= \Phi'(u_n)v = o(1) + o(1)  J'(u_n)v \\
		&=  o(1) + o(1) ( 2q_a(u_n,v) + 4 B_0(u_n^2,u_n v)),
\end{align*}
       and therefore 
        \begin{equation}
          \label{eq:J-consequence-mu-n-0}
        q_a(u_n,v) = o(1)- B_0(u_n^2,u_n v) \quad \text{as $n \to \infty$} \qquad \text{for all $v \in X$.}
        \end{equation}
        In particular, this implies that
        $$
        q_a(\hat u, v) = q_a(\hat u_n,v) + o(1) = o(1)- |u_n|_2^2 B_0({\hat{u}_n}^2, {\hat{u}_n} v)
        $$
        and therefore $
        B_0({\hat{u}_n}^2, {\hat{u}_n} v) \to 0.
        $ which implies that
        \begin{align*}
	B_1(\hat{u}_n^2,(\hat{u}_n - \hat{u})^2) &= B_1(\hat{u}_n^2,\hat{u}_n^2) - 2B_1(\hat{u}_n^2,\hat{u}_n \hat{u}) + B_1(\hat{u}_n^2,\hat{u}^2) \\
	&= V_1(\hat{u}_n) - 2B_1(\hat{u}_n^2,\hat{u}_n\hat{u}) + B_1(\hat{u}_n^2,\hat{u}_n\hat{u}) +o(1)\\
	&= V_2(\hat{u}_n) - B_1(\hat{u}_n^2,\hat{u}_n \hat{u}) +o(1) = V_2(\hat{u}_n) - \frac{1}{4} V_1'(\hat{u}_n)\hat{u} + o(1) = o(1).
\end{align*}
Then, from Lemma \ref{sec:comp-cond-2} we see that $\hat{u}_n \to \hat{u}$ in $X$, and therefore        
        $$
\frac{1}{4} V_0'(\hat{u})v= B_0(\hat{u}^2, \hat{u} v) = \lim_{n \to \infty}B_0({\hat{u}_n}^2, {\hat{u}_n} v) = 0 \qquad \text{for all $v \in X$,}
        $$
        which contradicts Lemma~\ref{lemma:V0'notzero}.
        
We thus conclude that $(u_n)_n$ remains bounded in $L^2(\R^2)$. Repeating the argument above with $u_n$ in place of $\hat{u}_n$, we see that there exists $u \in X \setminus \{0\}$ and $g_n \in G$, $n \in \N$ with
$g_n \ast_{\zeta}  u_n \to u$ weakly in $X$ and therefore strongly in $L^2(\R^N)$. Without loss of generality, we may, similarly as before, assume that $g_n = \id$ for all $n \in \N$ in the following, otherwise we replace $u_n$ by $g_n \ast_{\zeta} u_n \in X_0$. As before, we can now deduce, replacing $\hat{u}_n$ by $u_n$, that $\mu_n \to 0$ as $n \to \infty$, and therefore (\ref{eq:J-consequence-mu-n-0}) holds. We test this against $v = u_n - u$ to obtain
\begin{align*}
	o(1) &= \Phi'(u_n)(u_n - u) \\
	&= o(1) + q_a(u_n) - q_a(u) + \frac{1}{4} V_0'(u_n)(u_n - u) \\
	&= o(1) + |\nabla u_n|_2^2 - |\nabla u|_2^2 + \frac{1}{4} \left( (V_1^0)'(u_n)(u_n - u) - (V_2^0)'(u_n)(u_n - u) \right),
\end{align*}
and we have from \eqref{eq:2}
\[
\left| \frac{1}{4} (V_2^0)'(u_n)(u_n - u) \right| = \left| B_2^0(u_n^2, u_n(u_n - u)) \right| \leq |u_n|_{\frac{8}{3}}^3 |u_n - u|_{\frac{8}{3}} \to 0 \quad \text{as} \quad n \to \infty,
\]
and
\[
\frac{1}{4} (V_1^0)'(u_n)(u_n - u) = B_1^0(u_n^2, u_n(u_n - u)) = B_1^0(u_n^2, (u_n - u)^2) + B_1^0(u_n^2, u(u_n - u)),
\]
with
\[
B_1^0u_n^2, u(u_n - u)) \to 0 \quad \text{as} \quad n \to \infty .
\]
by Lemma \ref{lemma:B1weakconvergence}. We combine these estimates to obtain that
\[
o(1) = |\nabla u_n|_2^2 - |\nabla u|_2^2 + B_1^0(u_n^2, (u_n - u)^2) + o(1) \geq |\nabla u_n|_2^2 - |\nabla u|_2^2 + o(1) \geq o(1) \quad \text{as} \quad n \to \infty,
\]
which implies that \(|\nabla u_n|_2 \to |\nabla u|_2\) and \(B_1^0(u_n^2, (u_n - u)^2) \to 0\) as \(n \to \infty\). Hence \(\|u_n - u\|_X \to 0\) by Lemma~\ref{sec:comp-cond-2}, which implies that $u \in  \cN \cap X_0$ since $\cN \cap X_0$ is a closed subset of $X$.
\end{proof}

Combining Proposition~\ref{prop:PSC23} with Remark~\ref{remark:ceramiseq}, we obtain
\begin{corollary}
  \label{verif-NCG0}
Let $(G,G_0)$ be an admissible pair. Then the functional $\Phi:\mathcal{N} \to \R$ satisfies the \ref{NCG0}-condition for all $c \in \R \setminus \{0\}$.
\end{corollary}

\subsection{Completion of the proof of Theorem~\ref{thm:multiplicityN-}}
\label{sec:compl-proof-theor-1}
In this subsection we complete the proof of Theorem~\ref{thm:multiplicityN-}. We start by collecting some useful information on the family of $L^2$-invariant scaling operations
\begin{equation}\label{eq:dilation}
	T_t: X \to X, \qquad T_t(u)(x) = e^{-t} u (e^{-t}x), \qquad t \in \R.
\end{equation}
It is straightforward to see that
\begin{equation} \label{eq:normtransformed}
	|T_t(u)|_{2}^2 = |u|_{2}^2, \qquad |\nabla T_t(u)|_{2}^2 = e^{-2t} |\nabla u|_{2}^2, 
\end{equation}
and 
\begin{align}
	V_0(T_t(u)) &= e^{-4t} \int \int \log |x-y|u^2(e^{-t}x)u^2(e^{-t}y)\,dx dy \nonumber \\
	&=  \int \int \log \bigl(e^{t} |x-y|\bigr) u^2(x)u^2(y)\,dx dy =  V_0(u) +t |u|_{2}^4  \label{eq:transformedV0}
\end{align}
for $t>0$. In particular, for $u \in X$ with $|u|_{2}=1$ we have
$$
V_0(T_t(u))= V_0(u) +t
$$
We also note that, since we assume that $(G,G_0)$ is an admissible pair throughout this section, we have $G_0 \subset O(2)$ and therefore the scaling maps $T_t$ map $(G_0,\zeta)$-invariant functions to $(G_0,\zeta)$-invariant functions, i.e., we have $T_t(X_0) \subset X_0$.

We recall that our aim is to apply Theorem~\ref{c-k-to-infty-variant} to $M = \cN_-$ and the restriction $\Phi\big|_{\cN_-}$ of the functional $\Phi$ defined in (\ref{eq:def-Phi-Choquard}). To estimate the Ljusternik-Schnirelmann levels of $\Phi$ on $\cN_-$, we need an auxiliary variational problem. For this we consider  
$$ \Lambda := \{ u \in X : |u|_2 = 1, \, V_0(u) = 0 \} \qquad \text{and}\qquad \Lambda(c) := \{ u \in \Lambda : |\nabla u|_2^2 \leq c \}\quad \text{for $c>0$.}
$$
These sets are closed subsets of $X$. Now, using the identity \eqref{eq:transformedV0}, we may define the continuous and odd map
\begin{equation}
  \label{eq:fmap}
        f : X \setminus \{0\} \to \Lambda, \qquad f(u)= T_{-V_0(\hat{u})} \hat{u} \quad \text{with}\quad  \hat{u} := u/|u|_2.
\end{equation}
We first note the following.

\begin{proposition}\label{lemma:Sps}
The set $\Lambda(c)$ has finite Krasnoselskii genus (see Section \ref{sec:LS theory}) for every $c \in \R$.
\end{proposition}
\begin{proof}
  We need some preliminary considerations regarding the Schrödinger operator $-\Delta + W$ with $W(x)= \log(1+|x|)$. Since $W(x) \to \infty$ as $|x| \to \infty$, it follows from classical spectral theory that this operator is a selfadjoint operator in $L^2(\R^2)$ which admits a sequence of eigenvalues
  $$
  \lambda _1 < \lambda _2 \le \lambda_3 \le
  $$
  (counted with multiplicity) and the property that $\lambda_k \to \infty$ as $k \to \infty$. Let $(\phi_k)_k$ be a sequence of associated $L^2$-normalized eigenfunctions, and let $P_n$ be the orthogonal projection on the span $W_n$ of $\phi_1,\dots,\phi_n$ for $n \in \N$.
  We also let
  $$
  \widetilde \Lambda_n:=  \bigl \{u \in \Lambda \::\: P_n \bigl( g_{\beta(u)}^{-1}* u\bigr) \not = 0 \bigr\}
  $$
  Since the map $u \mapsto P_n \bigl( g_{\beta(u)}^{-1}* u\big)$ is odd and continuous and maps into the $n$-dimensional space $W_n$, it follows that $\gamma(\widetilde \Lambda_n) \le n$. Moreover, if $u \in \Lambda \setminus \widetilde \Lambda_n$, then $g_{\beta(u)}^{-1}* u$ is contained in the $L^2$-orthogonal complement of $W_n$, and thus we have
  \begin{align}
  S(u) + \int_{\R^2}\log(1+|x-\beta(u)|)u^2\,dx &= \int_{\R^2}\Bigl(\bigl|\nabla \bigl( g_{\beta(u)}^{-1}* u\bigr)\bigr|^2+ \log(1+|x|)\bigl|g_{\beta(u)}^{-1}* u\bigr|^2\Bigr)\,dx \nonumber\\
  &\ge \lambda_{n+1} |g_{\beta(u)}^{-1}* u|_2^2 = \lambda_{n+1} |u|_2^2 = \lambda_{n+1}.     \label{eq:eigenvalue-comp-estimate}
  \end{align}
  Now, to prove the proposition, we argue by contradiction and assume that $\gamma(\Lambda(c))= \infty$ for some $c>0$. Then we have
  $S^c \not \subset \widetilde \Lambda_n$ for every $n \in \N$, so there exists a sequence $(u_n)_n$ with $u_n \in \Lambda(c) \setminus \widetilde \Lambda_n$ for every $n$, which by (\ref{eq:eigenvalue-comp-estimate}) implies that
  \begin{equation}
    \label{eq:eigenvalue-comp-estimate-1}
  \int_{\R^2}\log(1+|x-\beta(u_n)|)u_n^2\,dx \ge \lambda_{n+1} - S(u_n) \ge \lambda_{n+1} -c \to \infty \quad \text{as $n \to \infty$.}
  \end{equation}
  On the other hand, since $|u_n|_2 = 1$ and $S(u_n) \le c$, the sequence $(u_n)_n$ is bounded in $H^1(\R^2)$. Then, according to Remark \ref{remark:H1boundedsequences}, up to a subsequence, we either have $u_n \to 0$ in $L^s(\R^2)$ for $s>2$ or there are translations $g_n = g_{id,b_n} \in E(2)$ such that $\hat{u}_n := g_n \ast u_n \rightharpoonup u \neq 0$ weakly in $H^1(\R^2)$. We claim that the second alternative is true. Indeed, let us assume by contradiction that $u_n \to 0$ strongly in $L^s(\R^2)$ for $s >2$. Then $V_2^\tau(u_n) \to 0$ for every $\tau \ge 0$ by (\ref{eq:3}), and therefore also $V_1^\tau(u_n) =V_0(u_n)- V_2^\tau(u_n)=- - V_2^\tau(u_n) \to 0$, contrary to the estimate $V_1^\tau(u_n) \geq \tau$ which follows from \eqref{eq:V1tau}.
	
Hence we have $\hat{u}_n := g_n \ast u_n \rightharpoonup u \neq 0$ weakly in $H^1(\R^2)$, as claimed. Moreover, for every $\tau \ge 0$ we have that $B_1^\tau(\hat{u}_n^2,\hat{u}_n^2) = V_1^\tau(\hat{u}_n)= - V_2^\tau(\hat{u}_n)$ remains bounded in $n$ by (\ref{eq:3}). Hence $(\hat{u}_n)_n$ is bounded in $X$ by Lemma~\ref{sec:comp-cond-2}, and it follows that $\hat{u}_n \rightharpoonup u$ weakly in $X$ and therefore strongly in $L^2(\R^2)$ by Lemma~\ref{sec:comp-cond}(i). As a consequence, we have
  $$
  \beta(\hat{u}_n) \to \beta(u) \qquad \text{as $n \to \infty$,}
  $$
so, in particular, the sequence $\beta(\hat{u}_n)$ is bounded. Hence, as in the proof of Lemma~\ref{lemma:scalarproduct}, we see that there exists constants $C_1, C_2>0$ with 
\begin{align*}
  |\hat{u}_n|_*^2 &= \int_{\R^2}\log(1+|x|){\hat u}_n^2\,dx \ge C_1 \int_{\R^2}\log(1+|x-\beta(\hat{u}_n)|){\hat u}_n^2\,dx - C_2 |\hat{u}_n|^2\\
                  &= C_1  \int_{\R^2}\log(1+|x-(\beta(u_n)+b_n)|)u_n^2(x-b_n)\,dx- C_2 |u_n|^2\\
  &= C_1 \int_{\R^2}\log(1+|x-\beta(u_n)|)u_n^2\,dx -C_2
\end{align*}
and therefore $|\hat{u}_n|_*^2 \to \infty$ by (\ref{eq:eigenvalue-comp-estimate-1}). This contradicts the fact that $\hat{u}_n$ is bounded in $X$. The contradiction shows that $\gamma(\Lambda(c))< \infty$ for every $c \in \R$, as claimed. 
\end{proof}

\begin{corollary}\label{cor:ak_unbounded}
Consider the Ljusternik-Schnirelmann type values for the functional $u \mapsto |\nabla u|_2^2$ on $\Lambda$ defined by 
\begin{equation}\label{eq:ak-Lambda}
	a_k := \inf\{c > 0 : \gamma(\Lambda(c)) \geq k\} \quad \in [-\infty,+\infty], \qquad k \in \N.
\end{equation}
Then we have $a_k \le a_{k+1}< \infty$ for every $k \in \N$ and $\lim \limits_{k \to \infty}a_k = +\infty$.
\end{corollary}

\begin{proof}
By definition (\ref{eq:ak-Lambda}), the values $a_k$ form an increasing sequence in $\R \cup \{\pm \infty\}$. Moreover, Lemma~\ref{lemma:Sps} implies that $a_k \to \infty$ as $k \to \infty$. To show that $a_k<\infty$ for every fixed $k \in \N$, it suffices to show that there exists $c > 0$ with $\gamma(\Lambda(c)) \geq k$. Indeed, let $\{\varphi_j\}_{j=1}^k \subset C^\infty_c(\R^2)$ be a set of functions with pairwise disjoint supports. Then, consider the compact and symmetric set
	$$ Z := \left\{ \sum_{j} s_j \varphi_j : \sum_{j} |s_j| = 1 \right\} \cong \mathbb{S}^{k-1}. $$
	Then, $\gamma(Z) = k$, and therefore
	$$
        \gamma(f(Z)) \geq \gamma(Z) = k,
        $$
	for the map $f$ defined in (\ref{eq:fmap}). Since $f(Z)$ is compact, we have $f(Z) \subset \Lambda(c)$ for some $c>0$, which implies that $\gamma(\Lambda(c)) \geq k$ and thus concludes the proof.	
\end{proof}

Now we recall that, in order to apply Theorem~\ref{c-k-to-infty-variant}, we consider the restriction $\Phi|_{M_0}:M_0 \to \R$ of the functional $\Phi$ to the set $M_0:= M \cap X_0= \cN_- \cap X_0$, the corresponding sublevels sets $\Phi^c= \{u \in M_0\::\: \Phi(u) \le c\}$ and the Ljusternik-Schnirelmann values
$$
c_k := \inf \{c>0 : \gamma(\Phi^c) \geq k\}.
$$
Following a standard argument, see \cite{palais_sym}, we see that the \textit{principle of symmetric criticality} holds, i.e. critical points of the restriction $\Phi:X_0 \to \R$ are critical points of the full functional $\Phi : X \to \R$.

\begin{lemma}\label{lemma:ckpositive}
	There exists $c>0$ with $ \gamma(\Phi^{c}) < \infty.$
\end{lemma}
\begin{proof}
  Since the sequence $(a_k)_k$ considered in Corollary~\ref{cor:ak_unbounded} satisfied $\lim \limits_{k \to \infty}= +\infty$, we may fix $k \in \N$ with $a_k > |a|_\infty$. We define $\alpha := (a_k +|a|_\infty)/2 > 0$, and we let $c > 0$ be such that $\gamma(\Phi^c) \geq k$. If no such $c$ exists, then the lemma is proved. Now, we recall that $\Phi^c \subset \mathcal{N}_-$, so we have $V_0(u) \leq 0$ for all $u \in \Phi^c$.
With $f$ defined in (\ref{eq:fmap}), we thus have 
$$
\gamma(f(\Phi^c)) \geq \gamma(\Phi^c) \geq k
$$
and therefore 
$$ \sup_{u \in \Phi^c}\|\nabla f(u)\|^2 \geq a_k.
$$
by definition of $a_k$. In particular, there exists $w \in \Phi^c$ with
$$
\alpha \leq \|\nabla f(w)\|_2^2 \|\nabla (w/|w|_2)\|_2^2,
$$
where we used \eqref{eq:normtransformed} and the fact that $V_0(w) \leq 0$ in the last inequality. It then follows that 
 $$ q_a(w) = | \nabla w|_2^2 + [w]_a^2 \geq (\alpha - |a|_\infty) |w|_2^2, $$
and for $\varepsilon > 1$ we have
 \begin{align*}
 	\varepsilon q_a(w) = q_a(w) + (\varepsilon -1) q_a(w) \geq (\alpha -\varepsilon |a|_\infty)|w|_2^2 + (\varepsilon -1) |\nabla w|_2^2.
 \end{align*}
 Fixing $\eps>1$ with $\alpha -\varepsilon |a|_\infty>0$ in this inequality, we thus find a constant $c_a > 0$, depending only on $a$ and $k$, such that $q_a(w) \geq c_a \|w\|_H^2$. On the other hand, we have, by Sobolev embeddings,
 $$ q_a(w) = -V_0(w) = V_2(w) - V_1(w) \leq V_2(w) \leq C |w|_{8/3}^4 \leq C \|w\|_H^4. $$
 Combining these inequalities, we find that $0<C' \leq \|w\|_H$ for some constant $C''$ depending only on $a$ and $k$. Thus, there is $C'' > 0$ depending only on $a$ and $k$ such that
 $$ c \geq \Phi(w) = q_a(w)/4 \geq C'' > 0. $$
 It follows that $c_k \geq C'' > 0$. Hence, for all $c \in (0,c_k)$ we conclude that $\gamma(\Phi^c) < k$ and then the lemma is valid for all $c_* \in (0,c_k)$.
\end{proof}

\begin{lemma}\label{lemma:genus_unbounded}
We have 
	$$ \lim_{c \to \infty} \gamma(\Phi^c) = \infty .$$
\end{lemma}
\begin{proof}
	Let $\varphi_1, \dots,\varphi_{k+1} \in  C_c^\infty(B_{1/2}(0))$ be $(G_0,\zeta)$-invariant functions with pairwise disjoint supports. Then, from the the definition of $V_0$ we see that $V_0(\varphi_j) < 0$. Now, we recall the transformation $T_t$ in \eqref{eq:dilation} and see that
	$$ q_a(T_t \varphi_j) = e^{-2t}| \nabla \varphi_j |_2^2 + [T_t \varphi_j]_a^2, $$
	where $|[T_t \varphi_j]_a^2| \leq |a|_\infty |\varphi_j|_2^2$, so this term is bounded for all $t$. Thus, we may choose, $t<0$ that satisfies $q_a(T_t \varphi_j) > 0$ for all $j =1,\ldots,k+1$ and since $t<0$ each $T_t \varphi_j$ is still compactly supported in $B_{1/2}(0)$. To keep the notation simple, we substitute the symbols $T_t \varphi_j$ with $\varphi_j$. This functions still have pairwise disjoint supports, and thus for $s_1,\dots,s_{k+1} \in \R$ with $\sum \limits_{j = 1}^{k+1} |s_j| = 1$ we have
	$$  q_a\left(\sum_{j } s_j \varphi_j\right) = \sum_j s_j^2 q_a(\varphi_j) > 0, $$
	and, since $\sum_{j } s_j\varphi_j$ is compactly supported in the ball $B_{1/2}(0)$ independent on the coefficients $(s_j)_j$, we have
	$$ V_0\left(\sum_{j } s_k \varphi_j \right) < 0.
        $$
Recalling the definition of the set $\mathcal{O}$ and the map $\sigma$ from Remark~\ref{remark:propertiesg}, we thus find that  $\sum_j s_j \varphi_j \in \mathcal{O}$. To conclude, we observe that
	$$ Z := \left\{ \sum_j s_j \varphi_j \in X : \sum_j |s_j| = 1 \right\} \cong  \mathbb{S}^{k} $$
	and we define 
	$$c:= \sup_{u \in \sigma(Z)} \Phi(u) < \infty.$$
	By the monotonicity of the genus, we see that
	$$ k = \gamma(Z) \leq \gamma(\sigma(Z)) \leq \gamma(\Phi^c),$$
	which proves the claim by letting $k \to \infty$.
\end{proof}

\begin{proof}[Proof of Theorem~\ref{thm:multiplicityN-} (completed)]
To verify the theorem we need only apply Theorem~\ref{c-k-to-infty-variant} to the restriction $\Phi : M_0 \to \R$ with the action $\ast = \ast_{\zeta}$ and the notation introduced in the beginning of this section. Here, we set $c_\infty = \infty$ and we fix $c_*>0$ according to Lemma \ref{lemma:ckpositive}. The conditions \ref{Gstar} and \ref{Ic'} are proved in Lemma \ref{example-conditions-CC-G-*} and Lemma \ref{i-g-compatible}, respectively. The remaining part of the assumptions $(i)-(iv)$ are verified in Lemmas \ref{lemma:dcomplete}, \ref{lemma:ckpositive}, \ref{lemma:genus_unbounded} and Proposition \ref{prop:PSC23}, respectively.

In the particular case $\einf_{\R^2} a > 0$ we can find a constant $C > 0$ such that
\begin{equation}\label{eq:q_a_bound}
	q_a(u) \geq C \|u \|_{H}^2 \geq 0 \qquad \text{for all } u \in X.
\end{equation}
Now, we see from \eqref{eq:3} and the Sobolev embeddings that
\begin{align*}
	\Phi'(u)u &= q_a(u) + V_0(u) \geq C \| u \|_{H}^2 - V_2(u) \geq C'( \| u \|_{H}^2 + |u|_{8/3}^4 )\\
	& \geq C'' (\| u\|_{H}^2 - \| u\|_{H}^4) 
\end{align*}
for some constants $C',C'' > 0$. Thereby $\Phi'(u)u > 0$ for all $u$ with small enough $H^1$-norm. Then, it follows that
$$ \inf_{u \in \mathcal{N}} \|u\|_{H} > 0.$$
This inequality and \eqref{eq:q_a_bound} imply that $\mathcal{N} = \mathcal{N}_-$. Now, since
$$ \Phi(u) = \frac{1}{4} q_a(u) \geq C \| u \|_{H}^2 \quad \forall u \in \mathcal{N},$$
we conclude that $\Phi$ is bounded from below on $\mathcal{N}$.

Then, in this case we apply Theorem \ref{c-k-to-infty-variant} to the restriction $\Phi: M_0 \to \R$ with $c_* = \inf_{u \in M_0} \Phi(u)$ and the rest of the data as above. Observe that $\Phi^{c_*} = K_{c_*}$. Theorem \ref{c-k-critical-value} shows that the values $c_k$ are critical for $k > \gamma(K_{c_*})$. To conclude, we only observe that if $k \leq \gamma(K_{c_*})$, then we have $c_k = \inf_{u \in M_0} \Phi(u)$ and from Proposition \ref{c-k-critical-value} we see that this is a critical value as well. Since $M_0 = \mathcal{N} \cap X_0$ contains all nontrivial $G_0$-invariant critical points, the result follows. 
\end{proof}

\bigskip
\noindent
{\bf Acknowledgments.}
O. Cabrera Chavez was fully supported by DAAD Research Grants - Doctoral Programmes in Germany, 2020/21 (57507871).
Silvia Cingolani is supported by PRIN PNRR  P2022YFAJH {\sl \lq\lq Linear and Nonlinear PDEs: New directions and applications''} (CUP H53D23008950001), and partially supported by INdAM-GNAMPA.

\end{document}